\documentclass[twoside]{article}

\usepackage{geometry}
 \usepackage{fancyhdr}
\usepackage{scrhack}

\usepackage[square,numbers,super,sort&compress]{natbib}

\makeatletter
\renewcommand\NAT@citesuper[3]{\ifNAT@swa
\if*#2*\else#2\NAT@spacechar\fi
\unskip\kern\p@\textsuperscript{\NAT@@open#1\if*#3*\else,\NAT@spacechar#3\fi\NAT@@close}%
   \else #1\fi\endgroup}
\makeatother

\usepackage[utopia]{mathdesign}

\usepackage{amsfonts, dsfont}
\usepackage{amsmath}
\usepackage{mathtools}
\usepackage{amsthm}
\usepackage{stmaryrd}
 
\usepackage{framed}
\usepackage{multicol}

\usepackage{color}
\usepackage{array}
\usepackage{booktabs}
\usepackage[final]{graphicx}
\usepackage{enumitem}
\usepackage{url}
\usepackage{varwidth}
\usepackage{hhline}

\usepackage{mathtools}

\usepackage[
  normal,
  font={small,color=black},
  labelfont=bf,
  margin=2em,
  labelsep=period]{caption}[2008/04/01]

\usepackage{caption}

\renewcommand{\kappa}{\varkappa}

 \providecommand{\keywords}[1]{\textbf{Keywords}  #1}

\newtheorem{lemma}{Lemma}[section]
\newtheorem{theorem}{Theorem}[section]
\newtheorem{remark}[lemma]{Remark}

\newcommand{\eps}{\varepsilon}
\newcommand{\sfrei}{}
\newcommand{\lang}{}

\begin{document}

\title{An
 edge-based pressure stabilisation technique for 
 finite elements on arbitrarily anisotropic meshes\protect\thanks{
 The author was supported 
 by the DFG Research Scholarship FR3935/1-1}}

\author{Stefan Frei\thanks{Department of Mathematics, University College London, Gower Street, WC1E 6BT, London, UK (s.frei@ucl.ac.uk)}}

\date{}

\maketitle

%
%
%


\abstract{
In this article, we analyse a stabilised equal-order finite
element approximation for the Stokes equations on
anisotropic meshes. In particular, we allow arbitrary anisotropies in a sub-domain, for example 
along the boundary of the domain, with the only condition that a maximum angle is fulfilled in each 
element.
This discretisation is
motivated by applications on moving domains
as arising e.g.$\,$in fluid-structure interaction or multiphase-flow problems.  
To deal with the anisotropies, we define
a modification of the original Continuous Interior Penalty stabilisation
approach. We show analytically the discrete stability of the method and convergence
of order ${\cal O}(h^{3/2})$ in the energy norm and ${\cal O}(h^{5/2})$ in the $L^2$-norm 
of the velocities.
We present numerical examples for a linear Stokes problem and for a non-linear fluid-structure
interaction problem, that substantiate the analytical results and show the capabilities of 
the approach.\\}


\maketitle

\section{Introduction}
\label{intro}

The motivation of this work is the finite element discretisation of the Stokes- or Navier-Stokes equations
on moving domains with finite elements. In order to impose boundary conditions and to obtain a certain
accuracy, it is necessary to resolve the evolving boundary within the discretisation.
The construction of fitted finite element meshes might not be straight-forward, however, when the domain 
$\Omega(t)$ changes from time step to time step. Constructing a new mesh in each time step can 
be expensive and projections to
the new mesh have to be chosen carefully in order to conserve the accuracy of the method. 
If the domain is 
not resolved accurately by the mesh, a severe reduction in the overall accuracy might result, 
see e.g.~\cite{Babuska1970} in the context of interface problems.

A simple method that avoids the decrease in accuracy as well as the computational cost to design new meshes
is the \textit{locally modified finite element
method} introduced by the author and Richter for elliptic interface problems~\cite{FreiRichter14}. 
The idea is to use a fixed coarse ``patch'' triangulation ${\cal T}_{2h}$ of a larger domain $D$
consisting of quadrilaterals, which is independent 
of the position of the boundary.
Based on this triangulation {\lang the patch elements are divided in such a way into 
either eight triangles or
four quadrilaterals, that} the boundary is resolved in a linear approximation. 
The degrees of freedom that lie outside of $\Omega$
can then be eliminated from the system. The \textit{locally modified finite element method} has been used by the author and 
co-workers~\cite{FreiRichterWick_Growth,FreiRichterWick_enumath1,FreiRichterWick_enumath2,DissFrei},
and by Langer \& Yang~\cite{LangerYang} for fluid-structure interaction problems. 
Holm et al.~\cite{Holmetal} and Gangl \& Langer~\cite{Gangletal} developed a
corresponding approach based on triangular patches, the latter work originating in the context 
of topology optimisation. 

The difficulty of this method lies in the highly anisotropic mesh cells, that can arise in the boundary 
region. Moreover, the type of anisotropy can change almost arbitrarily between neighbouring mesh cells.
This is in particular an issue in saddle-point problems, where the discrete spaces
have to satisfy a discrete \textit{inf-sup} condition. For many of the standard finite element pairs 
commonly used to approximate the Stokes or Navier-Stokes equations, the discrete \textit{inf-sup} condition 
is not robust with respect to anisotropies, see for example
the discussion in~\cite{AinsworthCoggins2000, BraackLubeRoehe2012}. An alternative is to 
use \textit{equal-order} 
finite elements in combination with a pressure stabilisation term that takes the anisotropies 
into account.

Highly anisotropic meshes arise also in very different applications. Obvious examples are
those where an anisotropic domain has to be discretised, e.g. when studying lubrication film 
dynamics~\cite{KnaufetalKugellager}. Anisotropic meshes are also used to resolve boundary or
interior layers, originating for example in convection-dominated problems. We refer to the textbook of 
Lin\ss~\cite{Linss2009} for an overview over some techniques to construct layer-adapted (so-called Shishkin
and Bakhvalov) meshes.
In the context of the Navier-Stokes equations, anisotropic meshes are used to resolve boundary layers
arising for moderate up to higher Reynolds numbers, see for 
example~\cite{PauliBehr2017, ApelKnoppLube2008, CodinaSoto2004, CastroHechtMohammadiPironneau1997}. There and in
many other applications,
anisotropic mesh refinement has proven a very efficient tool to reduce the computational costs, especially in
three space dimensions, see e.g.~\cite{Siebert1996,FormaggiaMichelettiPerotto2004,Richter2010,
LoseilleDervieuxAlauzet2010}. 

In this work, we will analyse the following linear Stokes model problem
\begin{align}
\begin{split}
 -\nu \Delta v + \nabla p &= f \\
 \text{div } v &=0
 \end{split} \quad \text{ in } \Omega,
 \begin{split}
 v &= 0 \quad \text{on } \Gamma^d \subset \partial\Omega,\\
 \nu \partial_n v - pn &= 0 \quad\text{on } \Gamma^n := \partial\Omega\setminus \Gamma^d,
   \end{split}\label{StrongStokes}
\end{align}
where we assume $\Gamma^d \neq \emptyset$. To simplify the error analysis, we will assume that 
$\Omega\subset\mathbb{R}^2$ is a convex polygonal domain. Both the restrictions {\lang to a convex polygon 
and to two space dimensions} are made only to simplify the {\lang presentation}.

Pressure stabilisation on anisotropic meshes has been studied for the \textit{Pressure-Stabilised
Petrov-Galerkin} (PSPG)
method~\cite{HughesFrancaBalestra1986}
by Apel, Knopp \& Lube~\cite{ApelKnoppLube2008} and for
\textit{Local Projection Stabilisations} (LPS)~\cite{BeckerBraack2001}
by Braack \& Richter~\cite{BraackRichter2005Enumath}.
For the analysis,
it seems however necessary for both methods that the change in anisotropy between 
neighbouring cells is bounded. 
This assumption can not be guaranteed for the locally modified finite element method, 
as we will explain in Section~\ref{sec.LMFEM}. Moreover, a coarser ``patch mesh'' necessary for 
the LPS method might not be available in the case of complex domains. Within the 
\textit{Galerkin Least Squares} (GLS) method~\cite{HughesFrancaetal1989}, optimal-order estimates for low-order schemes
have been obtained by Micheletti, Perotto \& Picasso~\cite{MichelettiPerottoPicasso2003}
without the assumption of a bounded change of anisotropy. {\sfrei Further works concerning GLS or PSPG pressure stabilisations
on anisotropic meshes include the references~\cite{FormaggiaMichelettiPerotto2004, MichelettiPerotto2007, PauliBehr2017}.}
Concerning the stabilisation of convection-dominated convection-diffusion equations, we refer to the survey article~\cite{JohnKnoblochSurvey}
and the textbook~\cite{RoosStynesTobiska}.

In this work, we will use a variant of the \textit{Continuous Interior Penalty} (CIP) 
stabilisation technique introduced by Burman \& Hansbo 
for convec\-tion-dif\-fu\-sion-reaction
problems~\cite{BurmanHansboConvDiff}. Later on, it has been used for 
pressure stabilisation within the Stokes~\cite{BurmanHansbo2006edge} and the 
Navier-Stokes equations~\cite{BurmanFernandezHansboNS}. The original CIP technique relies
on penalising jumps of the pressure gradient over element edges weighted by a factor 
$\mathcal{O}(h^s)$ for $s=2$ or $s=3$.
This is not applicable for the case of abrupt changes of anisotropy, however, 
as the cell {\lang sizes} of the two neighbouring cells can be very different.
Hence, in the boundary cells, we will use a weighted average of the pressure gradient 
instead of the jump terms.

Up to now, very few literature can be found for edge-based stabilisation 
techniques on anisotropic meshes. A few publications can be found for stabilisation of
convection-dominated CDR equations, see 
e.g.$\,$Micheletti \& Perotto~\cite{MichelettiPerotto} who designed a strategy for
anisotropic mesh refinement 
in an optimal control context. In these works, however, the jump terms are 
weighted by the edge size $h_{\tau}$, as the mesh size 
in direction normal to the edge might change significantly
from one cell to another. This is not an appropriate scaling for terms 
involving the normal derivative, that have to be used for pressure stabilisation.
To the best of the author's knowledge a detailed analysis of an edge-based pressure 
stabilisation method on arbitrarily anisotropic grids is not available in the 
literature yet.


The remainder of the article is organised as follows: In Section~\ref{sec.LMFEM}, we briefly 
review the locally modified finite element method, which is the main motivation for the present work.
In Section~\ref{sec.discgeneral}, we formulate the much more general assumptions on the finite element
method, that we will use in the analysis.
Next, in Section~\ref{sec.press}, we introduce the pressure stabilisation as well
as a projection operator for the discrete pressure gradient.
In Section~\ref{sec.propstab}, we show the properties
of the stabilisation term {\lang that will be needed in the analysis. Then,} we show the stability of discrete solutions
in Section~\ref{sec.stab} and derive a priori error estimates
in Section~\ref{sec.apriori}. Finally, we present {\sfrei three} numerical examples
in Section~\ref{sec.num}: First, we apply the method to {\sfrei solve stationary Stokes problems 
on extremely anisotropic meshes in Sections~\ref{sec.motstab} and \ref{sec.num_stat}}. Then, we study a non-stationary
{\sfrei and nonlinear} fluid-structure interaction problem with a moving {\lang interface} in Section~\ref{sec.num_fsi}.

\section{Discretisation}

In order to motivate the pressure stabilisation and some of the assumptions made below, we
start with a brief review of the \textit{locally modified finite element} method proposed by the author and 
Richter~\cite{FreiRichter14}. In Section~\ref{sec.discgeneral}, we will introduce 
the more general assumptions on the discretisation and the anisotropy of the mesh, that {\lang will be used} 
to prove stability and error estimates.

\subsection{The locally modified finite element method}
\label{sec.LMFEM}
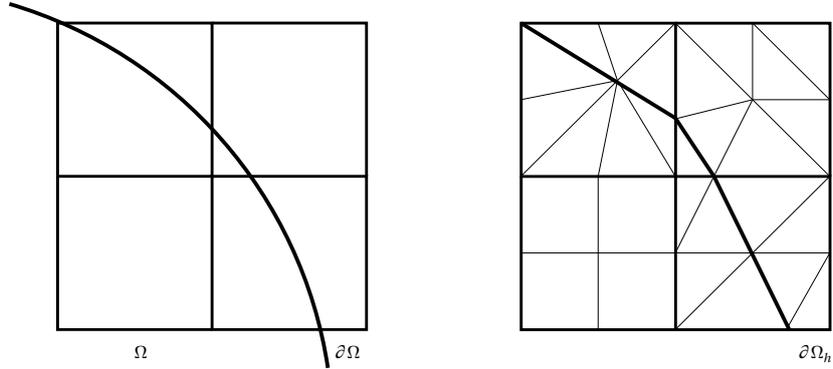
\begin{figure}[t]
  \centering
  \begin{picture}(0,0)%
\includegraphics{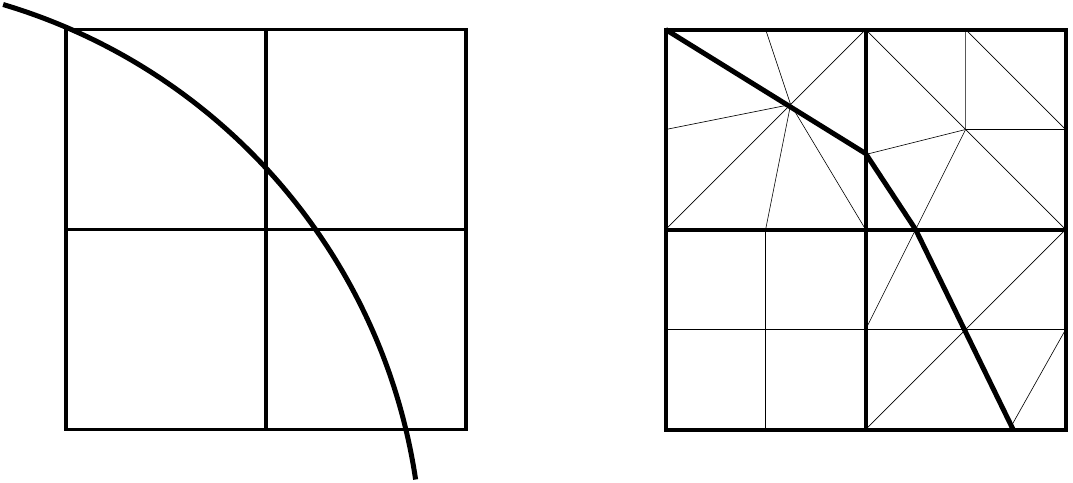}%
\end{picture}%
\setlength{\unitlength}{1579sp}%
\begingroup\makeatletter\ifx\SetFigFont\undefined%
\gdef\SetFigFont#1#2{%
  \fontsize{#1}{#2pt}%
  \selectfont}%
\fi\endgroup%
\begin{picture}(12836,5776)(413,-5799)
\put(12710,-5613){\makebox(0,0)[lb]{\smash{{\SetFigFont{7}{8.4}{\color[rgb]{0,0,0}$\partial\Omega_h$}%
}}}}
 \put(2401,-5611){\makebox(0,0)[lb]{\smash{{\SetFigFont{7}{8.4}{\color[rgb]{0,0,0}$\Omega$}%
 }}}}
\put(5506,-5611){\makebox(0,0)[lb]{\smash{{\SetFigFont{7}{8.4}{\color[rgb]{0,0,0}$\partial\Omega$}%
}}}}
\end{picture}%
  \caption{\textit{Left:} Triangulation ${\cal T}_{2h}$ of a domain $D$ that
   contains $\Omega$. \textit{Right:}
    Subdivision of the patches $P\in{\cal T}_{2h}$ such that the boundary $\partial\Omega$
    is resolved in a linear approximation by the discrete boundary $\partial\Omega_h$. Note that 
    the exterior cells 
    will not be used in the calculation.} 
  \label{fig:mesh}
\end{figure}

Let ${\cal T}_{2h}$ be a form- and shape-regular triangulation of a domain
$D\subset\mathbb{R}^2$ {\lang that contains $\Omega$} into open quadrilaterals. The triangulation ${\cal T}_{2h}$
does not necessarily resolve the domain
$\Omega$ and the boundary $\partial\Omega$ can
cut elements $P\in{\cal T}_{2h}$. 

%

Each patch $P$, which is not cut by the boundary $\partial\Omega$, is split into four quadrilaterals.
If the boundary goes through a patch $P$, we {\lang divide $P$} in such a way into
eight triangles that the boundary is resolved in a linear approximation.
To achieve this, we place degrees of freedom to the points of 
intersection $e_i\in \partial P\cap \partial\Omega, i=1,2$, see Figure\ref{fig:mesh} 
and the left sketch in Figure~\ref{fig:types}.

We define the finite element trial space $V_h^{\text{LMFEM}}$ as an
iso-parametric space on the triangulation ${\cal T}_{2h}$:
\[
V_h^{\text{LMFEM}} = \left\{\phi\in C(\bar D),\; \phi\circ T_P^{-1}\Big|_P\in
  \hat Q_P\text{ for all patches }P\in{\cal T}_{2h} \right\}, 
\]
where $T_P\in [\hat Q_P]^2$ is the {\lang (unique)} mapping between the reference patch
$\hat P=(0,1)^2$ and {\lang $P\in{\cal T}_{2h}$} such that
$T_P(\hat x_i) = x_i^P,\; i=1,\dots,9$
for the nine nodes $x_1^P,\dots,x_9^P$ of a patch. 
The local space $\hat{Q}_P$ consists of piecewise linear finite elements on eight triangles, 
if it is cut by the boundary and of piecewise bi-linear finite elements on four quadrilaterals
if $P\cap \partial\Omega=\emptyset$. Note that in both 
cases the discrete functions are linear on edges, such that 
%
%
%
mixing different element types does not affect the continuity of
the global finite element space. 
 
\begin{figure}[t]
  \begin{minipage}{0.5\textwidth}
  \centering
  \resizebox*{0.85\textwidth}{!}{
 \begin{picture}(0,0)%
\includegraphics{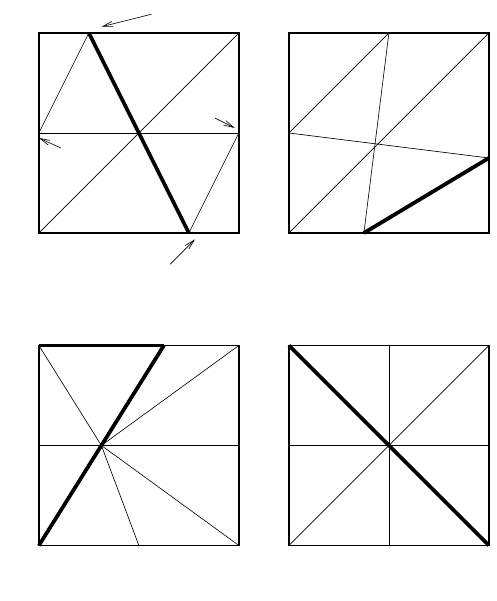}%
\end{picture}%
\setlength{\unitlength}{1579sp}%
\begingroup\makeatletter\ifx\SetFigFont\undefined%
\gdef\SetFigFont#1#2{%
  \fontsize{#1}{#2pt}%
  \selectfont}%
\fi\endgroup%
\begin{picture}(5955,7174)(736,-7175)
\put(1876,-3961){\makebox(0,0)[lb]{\smash{{\SetFigFont{6}{7.2}{\color[rgb]{0,0,0}$s$}%
}}}}
\put(5101,-7111){\makebox(0,0)[lb]{\smash{{\SetFigFont{5}{6.0}{\color[rgb]{0,0,0}$\textbf{D}$}%
}}}}
\put(2101,-7111){\makebox(0,0)[lb]{\smash{{\SetFigFont{5}{6.0}{\color[rgb]{0,0,0}$\textbf{C}$}%
}}}}
\put(2626,-136){\makebox(0,0)[lb]{\smash{{\SetFigFont{6}{7.2}{\color[rgb]{0,0,0}$e_3$}%
}}}}
\put(4351,-3061){\makebox(0,0)[lb]{\smash{{\SetFigFont{6}{7.2}{\color[rgb]{0,0,0}$r$}%
}}}}
\put(2626,-3361){\makebox(0,0)[lb]{\smash{{\SetFigFont{6}{7.2}{\color[rgb]{0,0,0}$e_1$}%
}}}}
\put(1351,-661){\makebox(0,0)[lb]{\smash{{\SetFigFont{6}{7.2}{\color[rgb]{0,0,0}$r$}%
}}}}
\put(1501,-1936){\makebox(0,0)[lb]{\smash{{\SetFigFont{6}{7.2}{\color[rgb]{0,0,0}$e_4$}%
}}}}
\put(1876,-2986){\makebox(0,0)[lb]{\smash{{\SetFigFont{6}{7.2}{\color[rgb]{0,0,0}$s$}%
}}}}
\put(2851,-1411){\makebox(0,0)[lb]{\smash{{\SetFigFont{6}{7.2}{\color[rgb]{0,0,0}$e_2$}%
}}}}
\put(6676,-2536){\makebox(0,0)[lb]{\smash{{\SetFigFont{6}{7.2}{\color[rgb]{0,0,0}$s$}%
}}}}
\put(2101,-3361){\makebox(0,0)[lb]{\smash{{\SetFigFont{5}{6.0}{\color[rgb]{0,0,0}$\textbf{A}$}%
}}}}
\put(5101,-3361){\makebox(0,0)[lb]{\smash{{\SetFigFont{5}{6.0}{\color[rgb]{0,0,0}$\textbf{B}$}%
}}}}
\end{picture}%
  }
  \end{minipage}
  \hfil
  \begin{minipage}{0.5\textwidth}
  \centering
  \resizebox*{0.7\textwidth}{!}{
  \begin{picture}(0,0)%
\includegraphics{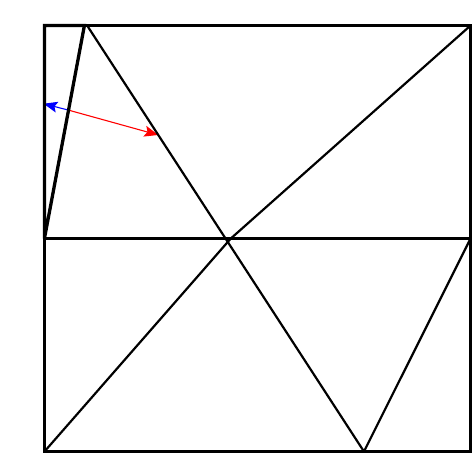}%
\end{picture}%
\setlength{\unitlength}{1381sp}%
\begingroup\makeatletter\ifx\SetFigFont\undefined%
\gdef\SetFigFont#1#2{%
  \fontsize{#1}{#2pt}%
  \selectfont}%
\fi\endgroup%
\begin{picture}(6496,6209)(736,-6412)
\put(1472,-342){\makebox(0,0)[lb]{\smash{{\SetFigFont{6}{7.2}{\color[rgb]{0,0,0}$r$}%
}}}}
\put(1610,-2797){\makebox(0,0)[lb]{\smash{{\SetFigFont{6}{7.2}{\color[rgb]{0,0,0}$e$}%
}}}}
\put(1867,-2368){\makebox(0,0)[lb]{\smash{{\SetFigFont{6}{7.2}{\color[rgb]{1,0,0}$h_n^2$}%
}}}}
\put(751,-1681){\makebox(0,0)[lb]{\smash{{\SetFigFont{6}{7.2}{\color[rgb]{0,0,1}$h_n^1$}%
}}}}
\end{picture}%
}
  \end{minipage}
  \caption{\textit{Left}: Four different types of cut patches A-D. The subdivision can be
    anisotropic with $r,s\in (0,1)$ arbitrary. \textit{Right:} Example for an edge $e$ with
    very different cell sizes $h_n^1, h_n^2$ in normal direction (type A and $r\to 0$).}\label{fig:types}
\end{figure}
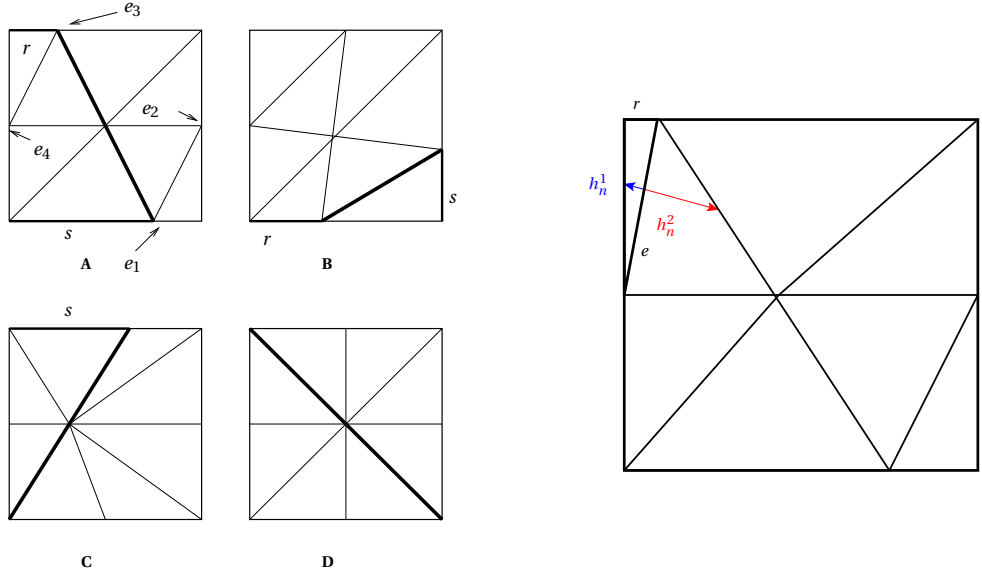

As the cut of the elements can be arbitrary with $r,s\to 0$ or $r,s\to
1$, the triangle's aspect ratio can be very large, considering  $h\to
0$ it is not necessarily bounded. Moreover, the cell size of neighbouring cells
can vary almost arbitrarily in the direction normal to the edge that is shared, 
see the right sketch of Figure~\ref{fig:types} for an example.
We can however guarantee, that the 
maximum angles in all triangles will be bounded away from
$180^\circ$:
\begin{lemma}[Maximum angle condition]\label{lemma:maxangle}
  All interior angles of the triangles shown in
  Figure~\ref{fig:types} are bounded by $144^\circ$ independent of
  $r,s\in (0,1)$.
\end{lemma}
\begin{proof}
 See Frei \& Richter~\cite{FreiRichter14}.
\end{proof}


Although we formulate the approach as a parametric approach on 
the patch mesh ${\cal T}_{2h}$, it is obvious that
the discretisation is equivalent to a mixed linear-bilinear discretisation on a finer 
mesh ${\cal T}_h$ that consists
of the sub-triangles and sub-quadrilaterals. With the help of the maximum angle condition, 
the following interpolation
estimate is well-known for the nodal interpolant $I_h u$
\begin{align*}
 \|u-I_h u\|_{L^2(\Omega)} + H \|u-I_h u\|_{H^1(\Omega)} \leq CH^2\|\nabla^2 u\|_\Omega,
\end{align*}
where $H$ denotes the {\lang maximum} element size of the (regular) patch grid~\cite{FreiRichter14}. An $H^1$-stable operator is given
by the Ritz projection, see Section~\ref{sec.H1stab}.

%
%
%

\subsection{Abstract setting and assumptions}
\label{sec.discgeneral}

We define a family of triangulations $\left({\cal T}_h\right)_{h>0}$ of the convex polygonal
domain $\Omega\subset\mathbb{R}^2$ into open triangles or quadrilaterals, such that the 
boundary $\partial\Omega$ is exactly resolved for all $h$. Motivated by the \textit{locally
modified finite element method}, we allow for mixed triangular-quadrilateral meshes. 
We remark that the restriction to
two dimensions is only made to simplify the presentation. The approach {\lang presented} here has a natural 
generalisation to three space dimensions and the theoretical analysis provided below 
generalises without significant differences. Moreover, the theoretical results can also be 
generalised to smooth domains $\Omega$, that are not necessarily convex. We will comment on
both generalisations in remarks.

We will write ${\cal X}_h$ for the set of
vertices, ${\cal E}_h$ for the set of edges
and ${\cal T}_h$ for the set of cells.
We assume that each triangulation ${\cal T}_h$ can be split into a part ${\cal T}_h^0$
and a part ${\cal T}_h^{\text{aniso}}$, such that the triangulation is quasi-uniform in 
$({\cal T}_h^0)_{h>0}$ in the usual sense, see for example~\cite{GrossmannRoos}. In 
particular, this includes a minimum 
and maximum angle condition for each element $K\in{\cal T}_h^0$ and the size of all edges belonging
to elements $K\in{\cal T}_h^0$ are of the same order of magnitude.
In ${\cal T}_h^{\text{aniso}}$ anisotropic cells are allowed. We relax the assumption 
of shape-regularity and assume a maximum angle condition only. We assume, however, that the maximum number
of neighbouring cells to a vertex in ${\cal T}_h$ is bounded independently of $h$.

We use the notation
$\Omega_h^0$ and $\Omega_h^{\text{aniso}}$ for
the region spanned by cells $K\in {\cal T}_h^0$ and $K\in {\cal T}_h^{\text{aniso}}$, 
respectively. 
Furthermore, we also split the set of faces into two parts: 
By ${\cal E}_h^0$, we denote all edges $e\in {\cal E}_h$ that lie between two regular cells
$K_1,K_2 \in {\cal T}_h^0$. By ${\cal E}_h^{\text{aniso}}$ we denote the edges that are edges of 
at least one element $K\in {\cal T}_h^{\text{aniso}}$. 
We denote that maximum size of an edge in ${\cal T}_h$ by
\begin{align*}
 H:= \max_{e \in {\cal E}_h} |e|.
\end{align*}
For the error analysis in Section~\ref{sec.apriori}, we will assume
that the area of the anisotropic part of the triangulation decreases linearly with $H$
\begin{align}
 |\Omega_h^{\text{aniso}}| = {\cal O}(H).\label{assAniso}
\end{align}
This will allow us to improve the optimal convergence order by a factor {\lang of order} ${\cal O}(H^{1/2})$.
In the case of the locally modified finite element method, we use ${\cal T}_h^0$ for all
cells of patches not cut by the interface, and ${\cal T}_h^{\text{aniso}}$ for all cells of 
patches that are cut by the interface. Assumption (\ref{assAniso}) follows then by the 
regularity of the patch mesh.

Let us now introduce the finite element spaces. By $P_r(\hat{K})$ and $Q_r(\hat{K})$ we denote
the usual polynomial spaces of degree $r$ on a reference
element $\hat{K}$. We define the spaces
\begin{align*}
V_h^r := \left\{\phi\in C(\Omega) \; \Big|\; (\phi\circ T^{-1})\in
  {\cal P}_K^r(\hat{K})\; \text{for } K \in{\cal T}_h\right\}, 
  \quad V_h^{r,0} := \left\{\phi\in V_h^r, \, \phi=0 \text{ on } \Gamma^d\right\},
\end{align*}
where $T\in {\cal P}_K^1(\hat{K})$ is a transformation from the reference element $\hat{K}$
to $K$ and 
\begin{align*}
{\cal P}_K^r(\hat{K}) := \begin{cases}
                        P_r(\hat{K}), \quad &K \text{ is a triangle},\\
                        Q_r(\hat{K}), &K \text{ is a quadrilateral}.
                        \end{cases}
\end{align*}

We restrict the analysis for simplicity to the case $r\leq 3$. Higher-order polynomials can be handled
as well, but as the approximation order
will be limited by the non-consistency of the stability term, they are not {\lang of interest} for the 
method presented here. 

Finally, we assume that the finite element space is spanned by a Lagrangian basis, i.e. there
{\lang exists} a set of Lagrange nodes ${\cal X}_h^L$ with ${\cal X}_h \subset {\cal X}_h^L$, such that  
each function in $v_h\in V_h^r$ can be represented as 
\begin{align*}
 v_h = \sum_{x_i\in {\cal X}_h^L} v_h(x_i) \phi_i,
\end{align*}
and the basis functions are defined via the relation $\phi_i(x_j)=\delta_{ij} (i,j=1,...,|{\cal X}_h^L|)$.

%


\section{Pressure stabilisation}
\label{sec.press}

The continuous variational formulation for the Stokes problem reads:
\textit{Find $v \in {\cal V}:=H^1_0(\Omega;\Gamma^d), p \in {\cal L} := L^2(\Omega)$ such that} 
\begin{align}
 A(v,p)(\phi,\psi)  &= (f,\phi)_{\Omega}\quad \forall \phi \in {\cal V}, \psi \in {\cal L},\label{ContStokes}
\end{align}
where
\begin{align*}
 A(v,p)(\phi,\psi):&= \nu (\nabla v,\nabla \phi)_{\Omega} - (p,\text{div } \phi)_{\Omega} + (\text{div } v,\psi)_{\Omega}.
\end{align*}

\subsection{Stabilisation}

{\sfrei 
For the discrete problem, we will use an edge-based pressure stabilisation technique. The standard continuous interior penalty technique 
for pressure stabilisation is given by
\begin{align}\label{defStandard}
 S_{cip}(p_h, \psi_h) = \gamma \sum_{e \in {\cal E}_h} h_{\tau}^3 \int_e [\partial_n p_h] [\partial_n \psi_h]\, do,
\end{align}
where $h_{\tau}$ denotes the length of an edge $e$ and $[\cdot]_e$ the jump operator across the edge $e$.
Heuristically, the weighting $h_e^3$ can be explained by the following observations.

Roughly speaking, in isotropic cells a factor $h^2$ is needed to compensate 
the normal derivatives $\partial_n \psi_h$ and $\partial_n p_h$, which grow with order ${\cal O}(h^{-1})$ when the cell size
$h$ gets small. In anisotropic cells, these derivatives grow with ${\cal O}(h_n^{-1})$ depending on the 
mesh size $h_n$ in direction normal to $e$. This motivates the choice $h_n^2$ instead of $h_{\tau}^2$
in an anisotropic context, 
see for example Braack \& Richter in the context of the LPS method~\cite{BraackRichter2005Enumath}. 
The third factor $h_{\tau}$ in (\ref{defStandard}) leads
in combination with the surface element of the integral (which is of order $h_{\tau}$) to a scaling with the cell size $K$
in isotropic cells. In the case of 
anisotropic cells, this factor $h_{\tau}$ should again be replaced with $h_n$.

These heuristic considerations motivate a stabilisation
\begin{align*}
 \tilde{S}_{cip}(p_h, \psi_h) = \gamma \sum_{e \in {\cal E}_h} h_n^3 \int_e [\partial_n p_h] [\partial_n \psi_h]\, do
\end{align*}
on anisotropic meshes.
As we allow abrupt changes in anisotropy in $\Omega_h^{\text{aniso}}$, $h_n$ can however vary strongly between neighbouring cells
and is therefore not well-defined on an edge $e$, see the right sketch in Figure 2.
Therefore, the stabilisation $\tilde{S}_{cip}$ can not 
be used in $\Omega_h^{\text{aniso}}$. Instead, we will use an average of the pressure gradients in the the anisotropic cells.
%
%
%
%

Precisely,} we define the stabilisation term by
\begin{align}
 S(p_h,\psi_h) := \gamma H^2 \Big(\sum_{e \in {\cal E}_h^{\text{aniso}}} 
 \int_e \{h_n \nabla p_h\cdot \nabla \psi_h\}_e \, do
 + \sum_{e \in {\cal E}_h^0} \int_e \{h_n\}_e [\nabla p_h]_e\cdot [\nabla \psi_h]_e \, do\Big),\label{defedgestab}
\end{align}
where $\gamma>0$ is a constant, $h_{n|K}:=\frac{|K|}{|e|}$ is the cell size in the direction normal 
to $e$ and
\begin{align*}
\{v_h\}_e:=\begin{cases}
 \frac{1}{2} \left(v_{h|K_1} + v_{h|K_2}\right), &e \not\subset \partial\Omega,\\
             v_{h|K_1}, &e \subset \partial\Omega.
           \end{cases}
\end{align*}
is the mean value {\lang of} the two cells $K_1,K_2$ sharing the edge $e$. {\sfrei
A mathematically more rigorous 
motivation for the choice of weights $H$ and $h_n$ and 
the averages instead of the jumps will be given within the error analysis in Section~\ref{sec.apriori}. 
Moreover, we will substantiate this analysis 
numerically by a comparison of the different variants in Section~\ref{sec.motstab}.
}

For later reference, we will denote the cell-wise contribution of an element $K\in{\cal T}_h$ by
\begin{align*}
 S_K(p_h,\psi_h) := \frac{\gamma}{2}H^2 
 \Big(\sum_{e\in {\cal E}_h^{\text{aniso}}, e\subset \partial K} 
 \int_e h_{n|K} \nabla p_{h|K}\cdot \nabla \psi_{h|K} \, do\; 
 + \sum_{e\in {\cal E}_h^0, e\subset \partial K} \int_e \{h_n\}_e [\nabla p_h]_e\cdot [\nabla \psi_h]_e \, do\Big).
\end{align*}
and the sum of the contributions from ``anisotropic'' and ``regular'' edges by
\begin{align*}
 S_h^{\text{aniso}}(p_h,\psi_h) 
 &:= \gamma H^2 \sum_{e \in {\cal E}_h^{\text{aniso}}} 
 \int_e \{h_n \nabla p_h\cdot \nabla \psi_h\}_e \, do,\\
 S_h^0(p_h,\psi_h) &:= \gamma H^2 \sum_{e \in {\cal E}_h^0} \int_e \{h_n\}_e [\nabla p_h]_e\cdot [\nabla \psi_h]_e \, do.
\end{align*}

\noindent The discrete formulation for the Stokes problem reads:
\textit{Find $v_h \in {\cal V}_h := \left(V_h^{r,0}\right)^2, p_h \in {\cal L}_h := V_h^r$ such that} 
\begin{align}
 A(v_h,p_h)(\phi_h,\psi_h) + S(p_h,\psi_h) &= (f,\phi_h)_{\Omega}\quad \forall \phi_h \in {\cal V}_h, \psi_h \in {\cal L}_h,\label{DiscStokes}
\end{align}
where
\begin{align*}
 A(v_h,p_h)(\phi_h,\psi_h):&= \nu (\nabla v_h,\nabla \phi_h)_{\Omega} - (p_h,\text{div } \phi_h)_{\Omega} + (\text{div } v_h,\psi_h)_{\Omega}.
\end{align*}

\subsection{A projection operator for the discrete pressure gradient}

Next, we introduce a projection that will be needed for the discontinuous gradient of $p_h$.
We denote the space of discontinuous functions of polynomial degree $r$ by
\[
V_h^{r,\text{dc}} = \left\{\phi:\Omega \to \mathbb{R}^2 \; \Big|\; (\phi\circ T^{-1})|_K\in
  P_r(\hat{K})\; \text{for } K \in {\cal T}_h \right\}. 
\]
Note that the gradient of a function $p_h \in V_h^r$ lies in $\left(V_h^{r,\text{dc}}\right)^2$. 

We define a projection 
$\tau_h: V_h^{r,\text{dc}} \to V_h^{r,0}$.
As $V_h^{r,0}$ is spanned by a Lagrangian basis, it is enough to specify the value 
of the projection $\tau_h v_h$ of $v_h\in V_h^{r,\text{dc}}$ 
in every Lagrange point $x_i\in {\cal X}_h^L$. 
A function $v_h \in V_h^{r,\text{dc}}$ might be
discontinuous in $x_i$, however, and thus a value $v_h(x_i)$ is not well-defined. 
Instead, we choose from the 
values $v_{h}|_K(x_i)$ of the cells $K$ surrounding $x_i$. 

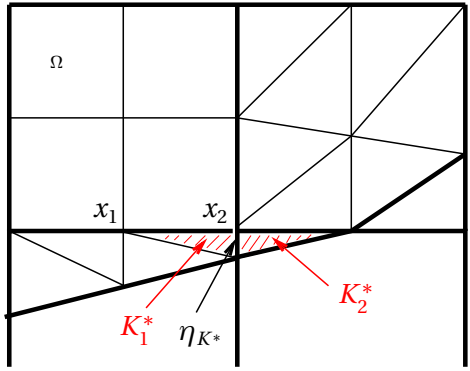
\begin{figure}
 \centering
\begin{picture}(0,0)%
\includegraphics{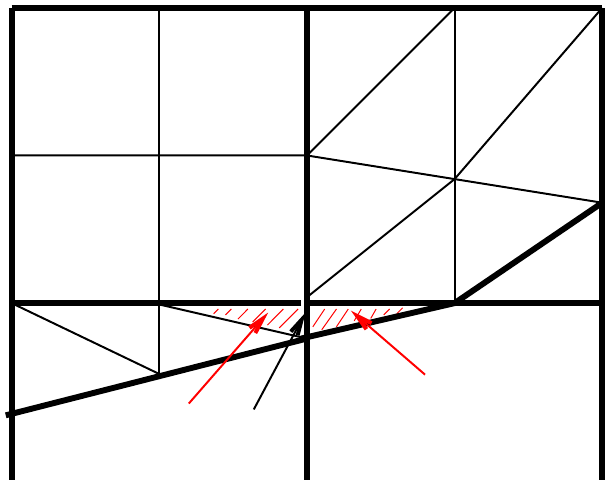}%
\end{picture}%
\setlength{\unitlength}{2486sp}%
\begingroup\makeatletter\ifx\SetFigFont\undefined%
\gdef\SetFigFont#1#2{%
  \fontsize{#1}{#2pt}%
  \selectfont}%
\fi\endgroup%
\begin{picture}(4633,3688)(2162,-5055)
\put(3916,-4741){\makebox(0,0)[lb]{\smash{{\SetFigFont{11}{13.2}{\color[rgb]{0,0,0}$\eta_{K^*}$}%
}}}}
\put(5491,-4381){\makebox(0,0)[lb]{\smash{{\SetFigFont{11}{13.2}{\color[rgb]{1,0,0}$K_2^*$}%
}}}}
\put(2656,-2041){\makebox(0,0)[lb]{\smash{{\SetFigFont{7}{8.4}{\color[rgb]{0,0,0}$\Omega$}%
}}}}
\put(3331,-4696){\makebox(0,0)[lb]{\smash{{\SetFigFont{11}{13.2}{\color[rgb]{1,0,0}$K_1^*$}%
}}}}
\put(3061,-3526){\makebox(0,0)[lb]{\smash{{\SetFigFont{11}{13.2}{\color[rgb]{0,0,0}$x_1$}%
}}}}
\put(4141,-3526){\makebox(0,0)[lb]{\smash{{\SetFigFont{11}{13.2}{\color[rgb]{0,0,0}$x_2$}%
}}}}
\end{picture}%

 \caption{\label{fig.projection}Cells $K_1^*, K_2^*$ corresponding to grid points 
 $x_1, x_2$ for the construction of the projection 
 $\tau_h$. In both $x_1$ and $x_2$ the shortest edge of the surrounding cells is the 
 edge $\eta_{K^*}$. While in $x_1$
 the choice of the cell $K_1^*$ is uniquely determined, we can choose either $K_1^*$ or $K_2^*$ in $x_2$.}
\end{figure}

Before we do this, let us introduce some notation. 
Let $\eta_{K,\min}$ be the shortest edge of a cell $K$. We denote its length by 
$h_{K,\min} = |\eta_{K,\min}|$. Moreover, we define the piece-wise constant function
\begin{align*}
 \tilde{h}_{\min|K} := \begin{cases}
                             h_{K,\min} \quad &K\in {\cal T}_h^{\text{aniso}},\\
                             H \quad &K\in {\cal T}_h^0,
                            \end{cases}
\end{align*}
which is approximately the length of the shortest edge of a cell $K$.
In a cell $K\in {\cal T}_h^0$ the minimal cell size $h_{K,\min}$ is
not necessarily equal to $H$, but of the same order of magnitude by assumption.

In $x_i$, we choose the value $v_{h}|_{K_i^*}(x_i)$ 
of a cell $K_i^*$ that possesses the smallest edge of the 
surrounding cells (in the sense of $\tilde{h}_{\min}$), 
see Figure~\ref{fig.projection} for an illustration.
The reason to use $\tilde{h}_{\min}$ instead of $h_{\min}$ is to give preference
to cells $K\in {\cal T}_h^{\text{aniso}}$. Precisely, we define
\begin{align}\label{dcProj}
 \tau_h v_h(x_i) = \begin{cases}
                    v_{h|K_i^*} (x_i) \quad &x_i\notin \partial\Omega,\\   
                    0 \quad &x_i\in \partial\Omega
                 \end{cases}
 \qquad \text{where } K_i^* = \underset{K\in {\cal T}_h,x_i\in \overline{K}}{\operatorname{argmin}} 
 \tilde{h}_{\min|K}.
\end{align}
{\lang If} this choice is not unique, we choose the value of
a cell $K\in {\cal T}_h^{\text{aniso}}$ if the vertex $x_i$ belongs to any. Otherwise we can pick
any of the cells.


We have the following stability result for the projection $\tau_h$:
\begin{lemma}
 Let $p_h \in V_h$ and $\tau_h$ the projection operator defined in (\ref{dcProj}).  
 It holds that 
 \begin{align}
  \left\| \nabla \tau_h \left(\tilde{h}_{\min}^2 p_h\right) \right\|_{\Omega} \leq CH \|\nabla p_h\|_{\Omega}, \label{stabtauh}
 \end{align}
 where $C$ is a constant that is
 independent of the position of the boundary.
\end{lemma}
\begin{proof}
Let $w_h := \tilde{h}_{\min}^2 p_h$. We start with an inverse inequality and 
use the definition of $\tau_h$
\begin{align}
 \|\nabla \tau_h w_h\|_K^2 \leq C h_{K,\text{min}}^{-2} \|\tau_h w_h\|_K^2
 \leq \sum_{x_i\in {\cal X}_h^K} C h_{K,\text{min}}^{-2} |\tau_h w_h (x_i)| \|\phi_i\|_K^2
 \leq \sum_{x_i\in {\cal X}_h^K} C h_{K,\text{min}}^{-2} |K| |w_{h|K_i^*}(x_i)|^2 \label{triquad}
\end{align}
where ${\cal X}_h^K$ is the set of all Lagrange points of a cell $K$ and $\phi_i$ are the corresponding Lagrangian 
basis functions.
By an inverse estimate, we obtain
\begin{align*}
 |w_{h|K_i^*}(x_i)|^2 =  \tilde{h}_{\min|K_i^*}^4 |\nabla p_{h|K_i^*}(x_i)|^2
 &\leq  \frac{C}{|K_i^*|} \tilde{h}_{\min|K_i^*}^4 \|\nabla p_h\|_{K_i^*}^2.
\end{align*}
Next, we note that $\tilde{h}_{\min|K_i^*} \leq C h_{K_i^*,\min}$, $|K| \leq h_{K,\max}^2$
and $|K_i^*| \geq C h_{K_i^*,\min} h_{K_i^*,\max}$.
In combination with (\ref{triquad}) this gives 
\begin{align*}
 \|\nabla \tau_h w_h\|_K^2 
 \leq C\sum_{i=1}^n \frac{|K|}{|K_i^*|} \frac{\tilde{h}_{K_i^*,\min}^4}{h_{K,\min}^2}\|\nabla p_h\|_{K_i^*}^2 
 &\leq C \frac{h_{K,\max}^2 h_{K_i^*,\min}^4 }{ h_{K_i^*,{\max}} 
 h_{K_i^*,{\min}}h_{K,\min}^2}\|\nabla p_h\|_{K_i^*}^2
 \,\leq\, CH^2 \|\nabla p_h\|_{K_i^*}^2.
\end{align*}
In the last step we have used that by definition $h_{K_i^*,\min}\leq h_{K,\min}, 
h_{K_i^*,\max}$ and $h_{K,\max}\leq H$.
\end{proof}

\subsection{Properties of the stabilisation}
\label{sec.propstab}

\noindent Next, we will show the properties of the stabilisation term that we will need in the analysis.
\begin{lemma}\label{lem.edgestab}
 Let $\psi_h \in V_h^r$ {\sfrei for $1\leq r\leq 3$}. There exists a constant $C>0$ independent of the boundary position 
 such that 
 {\lang the following lower bound holds
 for the set of cells }
 ${\cal T}_h^{\text{aniso}}$
 \begin{align}
  H^2 \sum_{K \in {\cal T}_h^{\text{aniso}}} \|\nabla \psi_h \|_{K}^2 \leq C S(\psi_h,\psi_h). \label{stabbelow}
 \end{align}
 The complete stabilisation term is bounded above by
 \begin{align}
  S(\psi_h,\psi_h) \leq CH^2 \|\nabla \psi_h \|_{\Omega}^2. \label{stabbound}
 \end{align}
 
 \noindent Furthermore, {\lang there holds} for a function $p_h\in V_h^{r,\text{dc}}$, the projection operator $\tau_h$ defined 
 in (\ref{dcProj}) {\lang and any} cell $K\in {\cal T}_h$ that
  \begin{align}
  \left\|\tilde{h}_{\min|K}^2\nabla p_h-\tau_h \left(\tilde{h}_{\min|K}^2 \nabla p_h\right)\right\|_K^2 
  \leq C \tilde{h}_{\min|K}^2 
  \sum_{L\in {\cal N}(K)} S_L(p_h,p_h), \label{projstab}
 \end{align}
 where ${\cal N}(K)$ denotes the set of neighbouring cells that share at least one common vertex with $K$.
 \end{lemma}
\begin{proof}
We start by showing that
\begin{align}
  c H^2 \|\nabla \psi_h \|_K^2 \leq S_K(\psi_h,\psi_h) \leq C H^2 \|\nabla \psi_h \|_K^2 \label{EquNorms}
\end{align}
for a cell $K\in {\cal T}_h^{\text{aniso}}$. This implies (\ref{stabbelow}) and the bound (\ref{stabbound})
for the cells belonging to ${\cal T}_h^{\text{aniso}}$. The inequalities (\ref{EquNorms}) follow by transformation 
to the reference element
and using equivalence of norms there. More precisely, we use that the functionals
  \begin{align*}
    s_1(\hat{\psi}_h) = \left(\sum_{\hat{e}\in\overline{\hat{K}}} \|\hat{\nabla} \hat{\psi}_h\|_{\hat{e}}^2\right)^{1/2} 
    \quad \text{and}\quad s_2(\hat{\psi}_h) = \|\hat{\nabla} \hat{\psi}_h\|_{\hat{K}}
  \end{align*}
  define both norms on the quotient space $Q_r(\hat{K}) / P_0$ for $r\leq 3$. The positivity 
  follows from the fact 
  that $s_i(\hat{\psi}_h) = 0$ implies $\hat{\psi}_h$=const in both 
  cases ($i=1,2$). This is obvious for $s_2$ and can be shown for $s_1$ by the following 
  argumentation: First, $s_1(\hat{\psi}_h) = 0$ implies that
  $\hat{\nabla} \hat{\psi}_h$ vanishes on the boundary of the reference element $\partial\hat{K}$.
  If $K$ is a quadrilateral, this means that $\hat{\psi}_h$ can be written as
  \begin{align*}
   \hat{\psi}_{h|\hat{K}}(\hat{x},\hat{y}) = \text{const} 
   + \hat{x}^2 \hat{y}^2 \left(1-\hat{x}\right)^2 \left(1-\hat{y}\right)^2 \hat{p}(\hat{x},\hat{y}),
  \end{align*}
  where $\hat{p}$ is a polynomial in $Q_{r-4}(\hat{K})$. As $\hat{\psi}_h \in Q_3(\hat{K})$, 
  we have $\hat{p}=0$ and thus
  $\hat{\psi}_{h|\hat{K}}=$ const. 
  In the case of a triangle, the same argumentation yields
    \begin{align*}
   \hat{\psi}_{h|\hat{K}}(\hat{x},\hat{y}) = \text{const} 
   + \hat{x}^2 \hat{y}^2 \left(1-\hat{x}-\hat{y}\right)^2 \hat{p}(\hat{x},\hat{y})
  \end{align*}
  with a polynomial $\hat{p}\in P_{r-6}$, which implies the positivity of $s_1$ even for
  polynomials up to order 5.
  
  {\lang The inequality} (\ref{stabbound}) follows when we prove the upper bound in (\ref{EquNorms}) also for 
  the cells $K\in {\cal T}_h^0$.
 Therefore, we estimate the
 jump terms very roughly by (note that $h_n\sim H$)
 \begin{align*}
  S_K(\psi_h,\psi_h) \leq \gamma H^3 \sum_{e\in\overline{K}} \|\nabla \psi_h\|_e^2.
 \end{align*}
 The bound (\ref{stabbound}) follows again by transformation to the reference element and 
 by the equivalence of norms on finite dimensional spaces.
 
  
To show (\ref{projstab}), we set $w_h = \nabla \tilde{h}_{\min}^2 p_h$ and estimate 
cell-wise for $K\in{\cal T}_h$
\begin{align*}
 \|w_h - \tau_h w_h\|_K^2  \leq |K| \sum_{x_i\in\overline{K}} |w_{h|K}(x_i) -\tau_h w_h(x_i)|^2.
\end{align*}
Let us first consider the case 
$x_i \in \partial\Omega$. We have
\begin{align*}
 |K| \,|w_{h|K}(x_i) -\tau_h w_h(x_i)|^2 = |K| \, |w_{h|K}(x_i)|^2.
\end{align*}
The inverse estimate $|w_h(x_i)|^2 \leq |e_i|^{-1} \|w_h\|_{L^2(e_i)}^2$ 
for an edge $e_i\subset \partial K$ 
with $x_i \in \overline{e}_i$ yields
\begin{align*}
 |K| \, |w_{h|K}(x_i)|^2 \leq C\frac{|K|}{|e_i|} \int_{e_i} |w_{h|K}|^2\, do 
 = C \int_{e_i} h_{n,K} h_{K,\min}^4 |\nabla p_{h|K}|^2\, do
 \leq C h_{K,\min}^2 S_K(p_h,p_h).
\end{align*}
For the case $x_i \notin \partial\Omega$, let us first note that $w_h(x_i) - \tau_h w_h(x_i)$ vanishes, 
when $x_i$ lies in the interior of $K$. For $x_i\in\partial K$,
we assume in a first step that $K$ and $K_i^*$ share a common edge $e_i$. 
An inverse estimate yields
\begin{align*}
|K| |w_{h|K}(x_i) -w_{h|K_i^*}(x_i)|^2  \leq C\frac{|K|}{|e_i|} \int_{e_i} |\,[w_h]_{e_i}|^2 \, do.
 \end{align*}
 If both $K$ and $K_i^*$ belong to ${\cal T}_h^0$, we have $\tilde{h}_{\min|K}= H$ and thus
 \begin{align}
  |K| |w_{h|K}(x_i) -w_{h|K_i^*}(x_i)|^2 
  \leq C \int_{e_i} h_n H^4 |\,[\nabla p_h]_{e_i} |^2 \, do \leq C H^2 S_K(p_h,p_h).\label{Omegah0}
 \end{align}
 If at least one of the cells belongs to ${\cal T}_h^{\text{aniso}}$, we estimate
 \begin{align}
  |K| |w_{h|K}(x_i) -w_{h|K_i^*}(x_i)|^2
  \leq C \int_{e_i}  \big\{h_{n|K} \tilde{h}_{\min}^4|\nabla p_h|^2\big\}_{e_i} \, do
  \leq C \tilde{h}_{K,\min}^2 \left(S_K(p_h,p_h) + S_{K_i^*}(p_h,p_h)\right)\label{triangles}
 \end{align}
as by definition $\tilde{h}_{K_i^*,\min} \leq \tilde{h}_{K,\min}$.
Finally, we have to consider the case that $K$ and $K_i^*$ do not share a common edge, 
but only the common point $x_i$.
First, we notice that if $K\in{\cal T}_h^{\text{aniso}}$, then $K_i^*\in{\cal T}_h^{\text{aniso}}$
by definition and we can estimate each of the summands separately using appropriate 
edges $e_K \subset \partial K$
and $e_{K_i^*} \subset \partial K_i^*$
 \begin{align}
 \begin{split}\label{ReasonForMeanValue}
   |K| |w_{h|K}(x_i) -w_{h|K_i^*}(x_i)|^2
  &\leq C |K|\,\left( |w_{h|K}(x_i)|^2 + |w_{h|K_i^*}(x_i)|^2\right) \\
  &\leq C \int_{e_K}  h_{K,\min}^4 h_{n} |\nabla p_h|^2 \, do 
  \,+\, C \frac{|K|}{|e_{K_i^*}|} \int_{e_{K_i^*}}  h_{K_i^*,\min}^4 |\nabla p_{h|K_i^*}|^2 \, do \\
  &\leq C h_{K,\min}^2 \left(S_K(p_h,p_h) + S_{K_i^*}(p_h,p_h)\right).
  \end{split}
 \end{align}
 In the last step, we have used that 
 $|K|\leq H^2, h_{K_i^*,\min}\leq |e_{K_i^*}|$
 and $h_{K_i^*,\min}\leq h_{n|K_i^*}$.
 
For $K\in {\cal T}_h^0$, we have $\tilde{h}_{\min,K}=H$. We split in the following way
\begin{align}
 |K| &|w_{h|K}(x_i) -w_{h|K_i^*}(x_i)|^2 
 \leq C|K| \left( |w_{h|K}(x_i) -w_{h|K_1}(x_i)|^2 
 + ... + |w_{h|K_n}(x_i) -w_{h|K_i^*}(x_i)|^2 \right), \label{split}
\end{align}
such that {\lang the cells share a common edge in each of the summands}. 
Now we apply the argumentations (\ref{Omegah0}) or (\ref{triangles}) to each of the summands.
\end{proof}

\begin{remark}{(Higher-order polynomials)}
The three inequalities (\ref{stabbelow}), (\ref{stabbound}) and (\ref{projstab}) 
are the properties of the stabilisation that we will exploit to show stability. The proof of 
Lemma~\ref{lem.edgestab} shows that the same results can be obtained for 
polynomial degrees up to order 5, if only triangles are used in $\Omega_h^{\text{aniso}}$, 
as in the locally modified finite element method. Higher polynomial degrees
can be controlled by using additionally higher-order derivatives in the stabilisation
term. As the approximation orders will be limited by the non-consistency of the stabilisation, 
however, high-order polynomials are not of interest for the method presented here.
\end{remark}


\section{Stability}
\label{sec.stab}
In this section, we prove a stability result. Therefore, and for the following error analysis, 
we will need $H^1$-stable projections.

\subsection{Ritz projection}
\label{sec.H1stab}

Defining an $H^1$-stable interpolation operator $\pi_h: H^1_0(\Omega; \Gamma^d)\to V_h^{r,0}$
that attains boundary values is not straight-forward. 
For the locally modified finite element method an $H^1$-stable operator could be obtained by defining
a standard $H^1$-stable interpolation
$i_{2h}: H^1(\Omega) \to V_{2h}$ of Cl\'ement~\cite{Clement75} or Scott-Zhang type~\cite{ScottZhang90}
onto the patch grid ${\cal T}_{2h}$ and an interpolation to ${\cal T}_h$. The $H^1$-stability follows
from the regularity
of the patch grid $\Omega_{2h}$. This interpolant will not fulfil
the boundary values, however, on boundary lines that lie in the interior of patches. 
A manipulation of this
operator is not straight-forward, as simply setting 
the desired boundary values in boundary nodes does not necessarily conserve 
the $H^1$-stability in anisotropic elements. 

For our purposes there is a simple solution, however. We can show that the Ritz 
projection operator $R_h: H^1_0(\Omega; \Gamma^d) \to V_h^{r,0}$
defined by
\begin{align}\label{defRitz}
 (\nabla R_h u,\nabla \phi_h)_{\Omega} = (\nabla u,\nabla \phi_h)_{\Omega} 
 \quad \forall \; \phi_h \in V_h^{r,0}
\end{align}
is $H^1$-stable. By definition, it also attains the boundary values. 

Moreover, we define a modified Ritz projection $\overline{R}_h: H^1(\Omega) \to V_h^r$
that conserves the global mean value of a function 
$u\in H^1(\Omega)$ instead of the Dirichlet boundary values. Therefore, we define the global mean value
by $\overline{u}=|\Omega|^{-1} \int_{\Omega} u \, dx$ and a finite element space by
\begin{align*}
  \overline{V}_h^{r} := \left\{\phi\in V_h^r, \, \overline{\phi}=0 \right\}.
\end{align*}
The modified Ritz projection is defined by: 
\textit{Find $\overline{R}_h u \in \overline{u} + \overline{V}_h^r$ such that}
\begin{align}\label{defModRitz}
 (\nabla \overline{R}_h u, \nabla \phi_h)_{\Omega} 
 = (\nabla u,\nabla \phi_h)_{\Omega} \quad \forall \phi_h \in \overline{V}_h^r.
\end{align}
We will use this projection for the pressure $p$ in the Stokes equations. The modification 
{\lang is necessary}
in the absence of Dirichlet
boundary conditions to obtain a well-defined operator.

We have the following approximation results for the Ritz projections.
\begin{lemma}
 \label{lem.RitzProj}
 Under the conditions of Section~\ref{sec.discgeneral},
 the Ritz projection defined in (\ref{defRitz}) is $H^1$-stable 
 \begin{align*}
 \|\nabla R_h u\|_{\Omega} \leq C \|\nabla u\|_\Omega \quad \forall u\in H^1_0(\Omega;\Gamma^d)
 \end{align*} 
 and we have the estimate
 \begin{align}\label{RitzH2}
 \|\nabla^j (u-R_h u) \|_{\Omega} \le CH^{s-j} \|u\|_{H^{s}(\Omega)}
 \end{align}
 for $j=0,1$ and $1\leq s\leq r+1$.
 The same results hold true for the modified Ritz projection $\overline{R}_h$ defined in 
 (\ref{defModRitz}).
\end{lemma}
\begin{proof}
The $H^1$-stability of $R_h$ follows by definition of the Ritz-projection (\ref{defRitz}) by testing with 
$\phi_h=R_h u$. 
For a quasi-uniform triangulation, the proof of
(\ref{RitzH2}) is standard and can be found in many textbooks.
Moreover, Babu\v{s}ka \& Az\'iz~\cite{BabuskaAziz1976} and Acosta \& Dur\'an~\cite{AcostaDuran2000} 
have shown for $s\geq 2$
that a maximum angle condition is sufficient to show (\ref{RitzH2}) for triangulations consisting of 
triangles and quadrilaterals, respectively. 
In particular, these works show besides (\ref{RitzH2}) the existence of an interpolation 
operator $I_h: H^2(\Omega) \to V_h^r$
that fulfils
\begin{align}
 \|\nabla^j (u-I_h u)\|_{\Omega} \leq CH^{2-j} \|u\|_{H^2(\Omega)}.\label{Interpol}
\end{align}
We only show the {\lang assertion} for $s=1$ here: For $j=1$, the estimate follows directly from the $H^1$-stability
of $R_h$. For the $L^2$-norm error estimate we use a dual problem: 
Let $z\in H^1_0(\Omega; \Gamma^d)$ be the solution of
\begin{align}\label{RitzDualP}
 (\nabla z,\nabla \phi)_{\Omega} = \left( \frac{u-R_h u}{\|u-R_h u\|},\phi\right)_{\Omega} 
 \quad \forall \phi \in H^1_0(\Omega; \Gamma^d).
\end{align}
As $\Omega$ is convex, $z$ lies in $H^2(\Omega)$ and $\|z\|_{H^2(\Omega)} \leq c$. 
Now we have by means of the definition of the Ritz projection, the interpolation estimate (\ref{Interpol})
and the Cauchy-Schwarz inequality
\begin{align*}
 \|u-R_h u\|_{\Omega} =(\nabla z,\nabla (u-R_h u))_{\Omega}
 = (\nabla (z-I_h z,\nabla (u-R_h u))_{\Omega}
 \,\leq\, CH \|\nabla^2 z\| \|\nabla (u-R_h u\|_\Omega
 &\,\leq\, CH \|\nabla u\|_\Omega.
\end{align*}
The results for the modified Ritz
 projection operator $\overline{R}_h$ can be shown with a very similar argumentation. Small
 modifications are necessary, whenever we have to test with a function with zero mean value. 
 To show the $H^1$-stability for example, we test (\ref{defModRitz}) with 
 $\phi_h = \overline{R}_h u - \overline{u}$ instead of $R_h u$.
\end{proof}

\subsection{Stability estimate}

Let us introduce the triple norm
\begin{align*}
 |||(v_h,p_h)|||_{\Omega} := \left(\nu \|\nabla v_h\|_{\Omega}^2 + 
 \|p_h\|_{\Omega}^2 + H^2 \|\nabla p_h\|_{\Omega}^2\right)^{1/2}.
\end{align*}

The argumentation used in the following proofs follows the lines 
of Burman \& Hansbo~\cite{BurmanHansbo2006edge} and Burman, Fern\'andez and 
Hansbo~\cite{BurmanFernandezHansboNS}.
Here, we have to modify their arguments in some parts, however, to account for the anisotropy of the mesh.
The main tool we use is the projection operator $\tau_h$ introduced in Section~\ref{sec.H1stab}.

\begin{theorem}\label{theo.StabAnisoSZ}
 Under the assumptions made in Section~\ref{sec.discgeneral} 
 it holds for $(v_h,p_h) \in {\cal V}_h \times {\cal L}_h$ with a constant $C$ 
 that is independent of the discretisation
 \begin{align*}
  |||(v_h,p_h)|||_{\Omega} \,\leq \, C \sup_{(\phi_h,\psi_h)\in {\cal V}_h\times {\cal L}_h} \frac{A(v_h,p_h)(\phi_h,\psi_h) + S(p_h,\psi_h)}{|||(\phi_h,\psi_h)|||_{\Omega}}. 
 \end{align*}
\end{theorem}
\begin{proof}
 At first we notice that
 \begin{align}
  A(v_h,p_h)(v_h,p_h) + S(p_h,p_h) = \nu \|\nabla v_h\|_{\Omega}^2 + S(p_h,p_h).\label{pstab1}
 \end{align}
Next, we derive a bound for the $L^2$-norm of the pressure $p_h$. 
Therefore, we use the surjectivity of the divergence operator (see e.g. Temam~\cite{Temam2000}) to define
a function $\tilde{v} \in H^1_0(\Omega)$ by
\begin{align*}
 (\text{div } \tilde{v}, \phi)_{\Omega} = -(p_h,\phi)_{\Omega} \quad \forall \phi \in L^2(\Omega).
\end{align*}
It holds that
\begin{align}
 \|\nabla \tilde{v} \|_{\Omega} \leq C \|p_h\|_{\Omega}.\label{tildev}
\end{align}
Using the test function $(\phi_h,\psi_h) = (\eps_1 R_h \tilde{v},0)$, where $R_h$ is the Ritz projection
operator introduced in Section~\ref{sec.H1stab} and $\eps_1>0$, we obtain
\begin{align}\label{tildevtest}
 A(v_h,p_h)(\eps_1 R_h \tilde{v},0) = \eps_1 \nu (\nabla v_h,\nabla R_h \tilde{v})_{\Omega} - \eps_1(p_h,\text{div} (R_h \tilde{v}))_{\Omega}. 
\end{align}
For the first term, we use the $H^1$-stability of the Ritz projection (Lemma~\ref{lem.RitzProj}) and (\ref{tildev}) to get
\begin{align*}
 \eps_1 \nu (\nabla v_h,\nabla R_h \tilde{v})_{\Omega} \geq -C\eps_1 \nu \|\nabla v_h\|_{\Omega} \|p_h\|_{\Omega} 
 \geq -\frac{\nu}{4} \|\nabla v_h\|_{\Omega}^2 - C\eps_1^2\|p_h\|_{\Omega}^2.
\end{align*}
For the second term in (\ref{tildevtest}), we add $\pm \tilde v$ and use (\ref{tildev}), integration by parts, 
the error estimate for the Ritz projection (Lemma~\ref{lem.RitzProj}) and Young's inequality
\begin{align}\label{ReasonForH}
\begin{split}
 -\eps_1(p_h,\text{div } R_h \tilde{v})_{\Omega} &= \eps_1(p_h,\text{div } (\tilde{v} -R_h \tilde{v}))_{\Omega} - \eps_1(p_h,\text{div}\,\tilde{v})_{\Omega}\\
 &= \eps_1(\nabla p_h,\tilde{v}- R_h \tilde{v})_{\Omega} + \eps_1 \|p_h\|_{\Omega}^2\\
 &\geq -C\eps_1 H \|\nabla p_h\|_{\Omega} \|p_h\|_{\Omega} + \eps_1 \|p_h\|_{\Omega}^2\\
 &\geq -C\eps_1 H^2 \|\nabla p_h\|_{\Omega}^2 + \frac{\eps_1}{2} \|p_h\|_{\Omega}^2.
 \end{split}
\end{align}
By combining the estimates, we have
\begin{align}
 A(v_h,p_h)(\eps_1 R_h \tilde{v},0) \geq -\frac{\nu}{4} \|\nabla v_h\|_{\Omega}^2 -C\eps_1 H^2 \|\nabla p_h\|_{\Omega}^2 + \frac{\eps_1}{4} \|p_h\|_{\Omega}^2.\label{pstab2}
\end{align}

Next, we will show a bound for the derivatives of $p_h$. Therefore, we test with the projection 
$\tau_h$ of the
discontinuous function $\tilde{h}_{\min}^2 \nabla p_h$ {\lang defined in (\ref{dcProj})}
\begin{align}
\begin{split}
 A(v_h,p_h)&\left(\eps_2 \tau_h\left(\tilde{h}_{\min}^2 \nabla p_h\right),0\right) 
 = \eps_2 \nu \left(\nabla v_h,\nabla \tau_h\left(\tilde{h}_{\min}^2 \nabla p_h\right)\right) 
 - \eps_2 \left(p_h, \text{div} \left(\tau_h\left(\tilde{h}_{\min}^2 \nabla p_h\right)\right)\right).\label{TestTau}
 \end{split}
\end{align}
We use the Cauchy-Schwarz inequality, the stability result (\ref{stabtauh}) 
for the projection $\tau_h$ and Young's inequality for the first part
\begin{align*}
 \eps_2 \nu \left(\nabla v_h,\nabla \tau_h\left(\tilde{h}_{\min}^2 \nabla p_h\right)\right) 
 &\geq -C\eps_2 H \nu \|\nabla v_h\|_{\Omega} \|\nabla p_h\|_{\Omega}
 \geq -C\eps_2 \nu \|\nabla v_h\|^2 - \frac{\eps_2 H^2}{8} \|\nabla p_h\|_{\Omega}^2.
\end{align*}
For the second part in (\ref{TestTau}), we apply integration by parts and 
insert $\pm \tilde{h}_{\min}^2 \nabla p_h$
 \begin{align*}
  - \eps_2 \Big(p_h, \text{div } \left(\tau_h\left(\tilde{h}_{\min}^2 \nabla p_h\right)\right)\Big) 
  &= \eps_2 \left(\nabla p_h, \tau_h\left(\tilde{h}_{\min}^2 \nabla p_h\right)\right)\\
  &= \eps_2 \left(\nabla p_h, \tau_h\left(\tilde{h}_{\min}^2 \nabla p_h\right) - \tilde{h}_{\min}^2 \nabla p_h\right) + \eps_2 \|\tilde{h}_{\min}\nabla p_h\|^2.  
 \end{align*}
For the first term, Lemma~\ref{lem.edgestab} guarantees in combination with Young's inequality
\begin{align*}
 \eps_2 \left(\nabla p_h, \tau_h\left(\tilde{h}_{\min}^2 \nabla p_h\right) 
 - \tilde{h}_{\min}^2 \nabla p_h\right)
 &\quad\geq -C \eps_2 \left( \sum_{K\in {\cal T}_h} \|\nabla p_h\|_K  \tilde{h}_{\min|K} \left(\sum_{L\in {\cal N}(K)} S_L(p_h,p_h)\right)^{1/2}\right)\\
 &\quad\geq -\frac{\eps_2}{2} \left\|\tilde{h}_{\min} \nabla p_h\right\|_{\Omega}^2 - C \eps_2 S(p_h,p_h).
\end{align*}
We have thus shown that
\begin{align}
 A(v_h,p_h)\left(\eps_2 \tau_h\left(\tilde{h}_{\min}^2 \nabla p_h\right),0\right) 
 \geq -C\eps_2 \nu \|\nabla v_h\|_{\Omega}^2 + \frac{\eps_2}{2} \left\|\tilde{h}_{\min} \nabla p_h\right\|_{\Omega}^2 - C \eps_2 S(p_h,p_h)
 - \frac{\eps_2 H^2}{8} \|\nabla p_h\|_{\Omega}^2.
 \label{pstab3}
\end{align}
Finally, we combine (\ref{pstab1}), (\ref{pstab2}) and (\ref{pstab3}) and choose $\eps_1 \ll \eps_2 \ll 1$
\begin{align*}
 A(v_h,p_h)(v_h+\eps_1 R_h \tilde{v}&+\eps_2 \tau_h\left(\tilde{h}_{\min}^2 \nabla p_h\right),p_h) + S(p_h,p_h)\\
 &\geq \frac{\nu}{2} \|\nabla v_h\|_{\Omega}^2 + \frac{1}{2} S(p_h,p_h) 
 + \frac{\eps_1}{4} \|p_h\|_{\Omega}^2
 + \frac{\eps_2}{2} \left\|\tilde{h}_{\min} \nabla p_h\right\|_{\Omega}^2 -\frac{\eps_2}{4}H^2 \|\nabla p_h\|_{\Omega}^2.
\end{align*}
For the last term, we note that $\tilde{h}_{\min|K}= H$ in all cells $K\in{\cal T}_h^0$. The contributions in 
the anisotropic elements $K\in{\cal T}_h^{\text{aniso}}$ can be estimated by
the stability term (see Lemma~\ref{lem.edgestab}). Thus, we have
\begin{align}\label{HnablapleqS}
 H^2 \|\nabla p_h\|_{\Omega}^2 \leq \left\|\tilde{h}_{\min} \nabla p_h\right\|_{\Omega}^2 
 + C S(p_h,p_h).
\end{align}
Altogether we have shown that
\begin{align*}
 |||(v_h,p_h)|||_{\Omega}^2 \leq C \left(A(v_h,p_h)\left(\phi_h,p_h\right) + S(p_h,p_h)\right)
\end{align*}
for 
\begin{align*}
\phi_h = v_h+\eps_1 R_h \tilde{v} + \eps_2 \tau_h\left(\tilde{h}_{\min}^2 \nabla p_h\right).
\end{align*}
Due to the stability results for the projection operators $\tau_h$ and $R_h$, we have $|||\phi_h,p_h)||| \leq C|||(v_h,p_h)|||$ and 
thus, the statement of the theorem is proven. 
\end{proof}

\begin{remark}{(Definition of the stabilisation term)}\label{rem.stabterm}
Let us comment on {\lang the form of the stabilisation term (\ref{defedgestab}),
in particular the use of averages} and the weights $H$. 
The reason {\lang to use averages} is to be able to control the term
$H^2 \|\nabla p_h\|_K$ in the anisotropic cells $K\in {\cal T}_h^{\text{aniso}}$ 
that appears in (\ref{ReasonForH}) by means of (\ref{stabbelow})
\begin{align*}
  H^2 \sum_{K \in {\cal T}_h^{\text{aniso}}} \|\nabla p_h \|_{K}^2 \leq C S(p_h,p_h).
 \end{align*}
In~\cite{BurmanFernandezHansboNS}
this was circumvented by testing with the $L^2$-projection $\pi_h \tilde{v}$ instead of the 
Ritz projection $R_h \tilde{v}$, which could be used to insert a projection $i_h \nabla p_h$ of
$\nabla p_h$ to $V_h^r$.
On anisotropic grids, the $L^2$-projection is {\lang however} not $H^1$-stable. Moreover, the argumentation used 
in~\cite{BurmanFernandezHansboNS} to control $(\nabla p_h - i_h (\nabla p_h))$ by the stabilisation
(which is similar to the argumentation (\ref{Omegah0}) and (\ref{split}) we used in $\Omega_h^0$)
relies on cells of the 
same size $h$ everywhere and can not be transferred to the situation considered here.

In the scaling of the cell-wise contributions, 
we have to use the size $H$ of the regular cells instead of the local cell sizes
$h_n$ and $h_{\tau}$, as this $H$ appears in (\ref{ReasonForH}) from
the approximation error of the Ritz projection. 
On structured grids with a bounded change of anisotropy
this estimate could be improved to
\begin{align*}
 \|\tilde{v}-i_h\tilde{v}\|_K \leq C \left( h_{\tau}^2 \|\partial_\tau \tilde{v}\|_K^2
+ h_n^2 \|\partial_n \tilde{v}\|_K^2
 \right)^{1/2}
\end{align*}
with an interpolation operator of Scott-Zhang type~\cite{ApelBuch}.
Then, the weights $H$ in the stability term could be replaced by $h_n$ and $h_{\tau}$, as
this stabilisation term $\tilde{S}$ would be an upper bound to
\begin{align*}
 \sum_{K\in {\cal T}_h^{\text{aniso}}} h_{K,\tau}^2 \|\partial_{\tau} p\|_K^2 
 + h_{K,n}^2 \| \partial_n p\|_K^2,
\end{align*}
which is needed in (\ref{HnablapleqS}).
\end{remark}

\section{A priori error analysis}
\label{sec.apriori}


We start with an estimate for the stabilisation term that we will need in the following:
\begin{lemma}\label{lem.Sh}
 Let $\psi_h\in {\cal L}_h$, {\lang $r\geq 1$} and $p\in H^r(\Omega) \cap W^{1,\infty}(\Omega)$. Under the conditions
 of Section~\ref{sec.discgeneral}, it holds
 with a constant $C$ 
 that is independent of the discretisation
 \begin{align*}
  S_h(\psi_h,\psi_h) \leq CH^2 \|\nabla (p-\psi_h)\|_\Omega^2 + 
  CH^{2r} \|p\|_{H^{r}(\Omega)}^2 + CH^3 \|p\|_{W^{1,\infty}(\Omega)}^2.
 \end{align*}
\end{lemma}
\begin{remark}
 We will use this lemma below for $\psi_h=p_h$ and $\psi_h=\overline{R}_h p$.
\end{remark}
\begin{proof}
 First, we note that for $r=1$ the estimate follows with Lemma~\ref{lem.edgestab} and the triangle inequality.
 For $r\geq 2$ we split into an anisotropic and a regular part. For the regular part, we 
 use that jumps of gradients
 over interior faces vanish for $p\in H^2(\Omega)$
 \begin{align*}
   S_h(\psi_h,\psi_h) \leq S_h^{\text{aniso}}(\psi_h,\psi_h) + S_h^0(\psi_h-p,\psi_h-p).
 \end{align*}
 For the anisotropic part we use (\ref{EquNorms}), the triangle inequality and the smallness
 of the sub-domain ${\cal T}_h^{\text{aniso}}$
 \begin{align*}
 S_h^{\text{aniso}}(\psi_h,\psi_h) \leq C H^2 \|\nabla \psi_h \|_{\Omega_h^{\text{aniso}}}^2
 &\leq C H^2 \left( \|\nabla (p-\psi_h) \|_{\Omega_h^{\text{aniso}}}^2 + 
 \|\nabla p \|_{\Omega_h^{\text{aniso}}}^2\right)\\
 &\leq C \left( H^2 \|\nabla (p-\psi_h) \|_{\Omega_h^{\text{aniso}}}^2 + 
 H^3 \|p \|_{W^{1,\infty}(\Omega_h^{\text{aniso}})}^2\right).
 \end{align*}
 For the regular part, we split once more, using the triangle and Young's inequality
 \begin{align*}
  S_h^0(\psi_h-p,\psi_h-p) \leq 2 \left(S_h^0(\psi_h-\overline{R}_h p,\psi_h-\overline{R}_h p) + S_h^0(\overline{R}_h p-p,\overline{R}_h p-p)\right).  
 \end{align*}
 We use (\ref{stabbound}) and the triangle inequality for the first part 
 (note that $h_{K,\min} \geq CH$ for $K\in{\cal T}_h^0$)
 \begin{align*}
  S_h^0(\psi_h-\overline{R}_h p,\psi_h-\overline{R}_h p) 
  &\leq CH^2 \|\nabla (\psi_h-\overline{R}_h p)\|_{\Omega}^2 
  \leq CH^2 \left(\|\nabla (\psi_h-p)\|_{\Omega}^2 + \|\nabla (p-\overline{R}_h p)\|_{\Omega}^2\right).
 \end{align*}
 For the second part, we apply the Poincar\'e-like estimate
 \begin{align*}
  \|[\psi]\|_e^2 \leq C \left( H^{-1} \|\psi\|_{K_1\cup K_2}^2 + H \|\nabla \psi\|_{K_1\cup K_2}^2\right),
 \end{align*}
  where $K_1,K_2$ denote the two cells surrounding $e$ 
  (see e.g. Bramble \& King~\cite{BrambleKing1994}, Ciarlet~\cite{Ciarlet1991}). Using Lemma~\ref{lem.RitzProj}
  in combination with an inverse estimate, we obtain 
 \begin{align*}
  S^0(\overline{R}_h p-p,\overline{R}_h p-p) &= 
  \gamma H^2 \sum_{e\in{\cal E}_h^0} h_n \int_e \big| \, [\nabla (\overline{R}_h p-p)]\big|^2\, d\text{o}\\
  &\leq CH^3 \sum_{K\in {\cal T}_h^0} \left( H^{-1} \|\nabla (\overline{R}_h p -p)\|_K^2 + H \|\nabla^2 (\overline{R}_h p -p)\|_K^2\right)\\
  &\leq CH^{2r} \|p\|_{H^{r}(\Omega)}^2.
 \end{align*}
 This completes the proof. 
\end{proof}

\noindent The a priori error analysis will be based on
the Galerkin orthogonality
\begin{align}\label{FullGalOrth}
   A(v-v_h,p-p_h)(\phi_h,\psi_h) - S(p_h,\psi_h) &= 0 \quad 
  \forall\phi_h\in{\cal V}_h, \psi_h\in {\cal L}_h.
\end{align}
We have the following result.

\begin{theorem}\label{theo.conv}
 Let {\sfrei $1\leq r\leq 3$} and let $(v,p) \in \left(H^{r+1}(\Omega) \times \left(H^r(\Omega) \cap W^{1,\infty}(\Omega)\right)\right)$ 
 {\lang and $(v_h,p_h) \in {\cal V}_h^r \times {\cal V}_h^r$ the solution of (\ref{ContStokes}) and (\ref{DiscStokes}), 
 respectively}. Under the conditions
 of Section~\ref{sec.discgeneral}
 it holds that
 \begin{align}
  |||(v-v_h,p-p_h)|||_{\Omega} 
  \leq CH^{\min\{r;3/2\}} \left( \| v \|_{H^{r+1}(\Omega)} + \| p \|_{H^{r}(\Omega)}
  + \| p \|_{W^{1,\infty}(\Omega)} \right). \label{edgeenergy}
 \end{align} 
 Furthermore, we have for the $L^2$-norm error of the velocities
 \begin{align*}
  \|v-v_h\|_{\Omega} \leq CH^{\min\{r+1;5/2\}} \left( \| v \|_{H^{r+1}(\Omega)} + \| p \|_{H^{r}(\Omega)}
  + \| p \|_{W^{1,\infty}(\Omega)} \right).
 \end{align*}
\end{theorem}
\begin{proof}
 We prove the energy norm estimate first. Therefore, we split the error into a projection and a discrete part
 \begin{align*}
  |||(v-v_h)&,(p-p_h)|||_{\Omega}
  \leq |||(v-R_h v),(p-R_hp)|||_{\Omega} + |||(R_h v-v_h),(\overline{R}_h p-p_h)|||_{\Omega}.
 \end{align*}
By Lemma~\ref{lem.RitzProj} we get the following bound for the Ritz projections
\begin{align*}
 |||(v-R_h v), (p-\overline{R}_h p)|||_{\Omega} \leq CH^{r} \left(\| p \|_{H^{r}(\Omega)} 
 + \|v\|_{H^{r+1}(\Omega)} \right).
\end{align*}
For the discrete part, Theorem~\ref{theo.StabAnisoSZ} yields
  \begin{align*}
  |||(R_h v -v_h,\overline{R}_h  p -p_h)|||_{\Omega}
  \leq C \sup_{(\phi_h,\psi_h)\in {\cal V}_h\times {\cal L}_h} \frac{A(R_h v - v_h,\overline{R}_h p - p_h)(\phi_h,\psi_h) + S(\overline{R}_h  p - p_h,\psi_h)}{|||(\phi_h,\psi_h)|||}. 
 \end{align*}
We use the Galerkin orthogonality (\ref{FullGalOrth})
\begin{align}
\label{GalOrthIsUsed1}
 A(R_h v - v_h,\overline{R}_h p - p_h)(\phi_h,\psi_h) + S(\overline{R}_h  p - p_h,\psi_h)
 = A(R_h v - v,\overline{R}_h p - p)(\phi_h,\psi_h) + S(\overline{R}_h  p,\psi_h)
\end{align}
and by means of the the Cauchy-Schwarz inequality, it follows that
\begin{align*}
 A(R_h v - v,&\overline{R}_h  p - p)(\phi_h,\psi_h) + S(\overline{R}_h  p,\psi_h) \\
 &\leq C \left(\|\nabla (R_h v - v)\|_{\Omega} + \|\overline{R}_h  p - p \|_{\Omega} + S^{1/2}(\overline{R}_h p,\overline{R}_h p) \right) |||(\phi_h,\psi_h)|||.
\end{align*}
With the help of Lemma~\ref{lem.Sh} and the estimates for the Ritz projection (\ref{RitzH2}), we obtain
\begin{align*}
 |||(R_h v -v_h,\overline{R}_h  p -p_h)|||_{\Omega} \leq CH^{\min\{r;3/2\}} 
 \left(\|v\|_{H^{r+1}(\Omega)} + \|p \|_{H^{r}(\Omega)} 
 + \|p\|_{W^{1,\infty}(\Omega)}\right),
\end{align*}
which proves (\ref{edgeenergy}).

To show the $L^2$-norm estimate, we make use of a dual problem. 
Let $(v^*,p^*) \in (H^1_0(\Omega;\Gamma^d) \times L^2(\Omega))$
the solution of
\begin{align}
 A(\phi,\psi)(v^*,p^*) = (v-v_h,\phi)_\Omega.\label{dualPStokes}
\end{align}
 {\lang As $\Omega$ is assumed to be} a convex polygon, we have
\begin{align}
 \|v^*\|_{H^2(\Omega)} + \|p^*\|_{H^1(\Omega)} \leq C \|v-v_h\|_{\Omega}.\label{RegDual}
\end{align}
We test with $\phi= v-v_h, \psi= p-p_h$ and use the Galerkin orthogonality (\ref{FullGalOrth}) 
\begin{align}
\begin{split}
\|v-v_h\|_\Omega^2 &=  A(v-v_h,p-p_h)(v^*,p^*)\\
&= A(v-v_h,p-p_h)(v^*- R_h v^*,p^*-\overline{R}_h p^*) 
+ S(p_h, \overline{R}_h p^*).\label{GalOrthIsUsed2}
\end{split}
\end{align}
Using the regularity of the dual solution (\ref{RegDual}), we obtain for the first part
\begin{align*}
A(v-v_h,p-p_h)(v^*- R_h v^*,p^*-\overline{R}_h p^*)
&\leq CH |||(v-v_h, p-p_h||| \, \|v-v_h\|_\Omega.
\end{align*}
For the stabilisation term, we have with Lemma~\ref{lem.edgestab} and the $H^1$-stability of 
the Ritz projection 
\begin{align*}
 S(p_h,\overline{R}_h  p^*) \leq S(p_h,p_h)^{1/2} S(\overline{R}_h  p^*,\overline{R}_h  p^*)^{1/2} 
 \,\leq\, CH S(p_h,p_h)^{1/2} \|\nabla p^*\|_{\Omega}
 \,\leq\, CH S(p_h,p_h)^{1/2} \|v-v_h\|_{\Omega}
\end{align*}
For the first term, Lemma~\ref{lem.Sh} gives us
\begin{align*}
S(p_h,p_h)^{1/2} &\leq C \left(H \|\nabla (p-p_h)\|_\Omega + 
 CH^{r} \|p\|_{H^{r}(\Omega)} + H^{3/2} \|p\|_{W^{1,\infty}(\Omega)}\right)\\
  &\leq C \left(|||(v-v_h,p-p_h)||| + H^{r} \|p\|_{H^{r}(\Omega)} + 
  H^{3/2} \|p\|_{W^{1,\infty}(\Omega)}\right). 
\end{align*}
%
This completes the proof.
\end{proof}


{\sfrei
\begin{remark}{(Polynomial degrees)}
 Theorem~\ref{theo.conv} shows in particular the optimal convergence orders for 
 first-order polynomials $r=1$. For quadratic elements the convergence orders
 are improved to ${\cal O}(H^{3/2})$ in the energy norm and ${\cal O}(H^{5/2})$ in 
 the $L^2$-norm of velocities.  Due to the non-consistency of the stabilisation, 
 no further improvement of the convergence orders is 
 achieved for higher-order polynomials.
\end{remark}
}

\begin{remark}{(Larger anisotropic region)}
 We remark that the stability result in {\lang Theorem~\ref{theo.StabAnisoSZ}} holds independently of the 
 {\lang smallness assumption for}
 $\Omega_h^{\text{aniso}}$ (\ref{assAniso}). {\lang Without this assumption} the convergence rates in 
 Theorem~\ref{theo.conv} {\lang are} bounded by ${\cal O}(H^{\min\{r;1\}})$ in the energy norm 
 and by ${\cal O}(H^{\min\{r+1;2\}})$ in the $L^2$-norm of the velocity. The same results hold true, 
 if the pressure $p$ is {\lang only in $H^1(\Omega)$, but not in $W^{1,\infty}(\Omega)$}.
\end{remark}

\begin{remark}{(Smooth domains)}\label{rem.smoothDomain}
 The same results can be shown for a smooth domain $\Omega$, that is not necessarily convex, 
 instead of a convex polygon. To achieve convergence orders $r\geq 1$ in the energy norm, 
 iso-parametric finite elements must be used in the boundary cells
 to obtain a higher-order boundary approximation.
 Key to the proof is the estimation of certain integrals over the {\lang regions} $\Omega\setminus\Omega_h$
 and $\Omega_h\setminus\Omega$, where $\Omega_h$ denotes the ``discrete'' domain spanned by the 
 mesh cells, and the derivation of a perturbed Galerkin orthogonality. For the details, we refer to
 Richter~\cite{RichterBuch} or Ciarlet~\cite{CiarletBuch}.
\end{remark}

\begin{remark}{(3 space dimensions)}
The argumentation can be easily generalised to $\Omega\subset \mathbb{R}^3$ using the analogously defined
stabilisation term {\lang to} (\ref{defedgestab}). Here, ${\cal E}_h$ denotes the set of faces instead of edges.
\end{remark}

\section{Numerical examples}
\label{sec.num}

{\sfrei
In the following we will present three numerical examples to substantiate the analytical findings and 
to show the capabilities of the approach. First, we motivate the form of the 
pressure stabilisation term in Section~\ref{sec.motstab} by comparing it with different alternatives including 
the standard CIP pressure stabilisation in an example with
alternating isotropic and anisotropic cells.
In Section~\ref{sec.num_stat}, we show that the stabilisation can be used for all kind of different 
anisotropies that arise using the locally modified finite element method. Finally, we apply the pressure stabilisation in a non-stationary 
and non-linear fluid-structure interaction problem with a moving interface in Section~\ref{sec.num_fsi}.
All examples include extremely anisotropic cells with aspect ratios $\kappa_K\geq 1000$.
}

{\sfrei
\subsection{Example 1: Comparison of different edge-based pressure stabilisation terms}
\label{sec.motstab}

In a first example, we would like to motivate the form of the stabilisation term in the anisotropic cells numerically. 
Therefore, we discretise the unit square $\Omega=[-1,1]^2$ with anisotropic cells without a bounded change in anisotropy. To be precise,
we define the cell sizes in vertical direction in an alternating way to be $h_y=H/1000$ and $999 H/1000$, while 
the cell sizes in horizontal direction are uniform $h_x=H/2$. A sketch of a resulting 
coarse grid is given 
in Figure~\ref{fig.anisomesh}. 

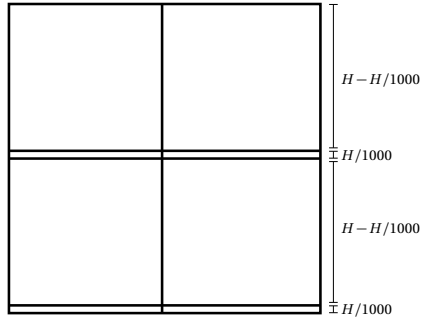
\begin{figure}
\centering
 \resizebox*{0.3\textwidth}{!}{
\begin{picture}(0,0)%
\includegraphics{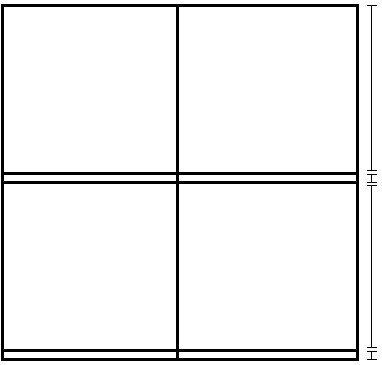}%
\end{picture}%
\setlength{\unitlength}{1243sp}%
\begingroup\makeatletter\ifx\SetFigFont\undefined%
\gdef\SetFigFont#1#2{%
  \fontsize{#1}{#2pt}%
  \selectfont}%
\fi\endgroup%
\begin{picture}(5819,5522)(1757,-6439)
\put(7561,-6384){\makebox(0,0)[lb]{\smash{{\SetFigFont{5}{6.0}{\color[rgb]{0,0,0}$H/1000$}%
}}}}
\put(7561,-3684){\makebox(0,0)[lb]{\smash{{\SetFigFont{5}{6.0}{\color[rgb]{0,0,0}$H/1000$}%
}}}}
\put(7561,-2356){\makebox(0,0)[lb]{\smash{{\SetFigFont{5}{6.0}{\color[rgb]{0,0,0}$H - H/1000$}%
}}}}
\put(7561,-4966){\makebox(0,0)[lb]{\smash{{\SetFigFont{5}{6.0}{\color[rgb]{0,0,0}$H - H/1000$}%
}}}}
\end{picture}%
 }
 \caption{\sfrei\label{fig.anisomesh} Sketch of a very coarse mesh with alternating cell sizes $h_y$ in vertical direction. 
 Note that the anisotropies in the sketch are less 
 extreme than they actually are for visualisation purposes.}
\end{figure}

We consider the Stokes equations given in (\ref{ContStokes}) with viscosity $\nu=1$ and impose a \textit{do-nothing} boundary condition 
on the right boundary: $\partial_n v - p n =  0$. Furthermore, we specify non-homogeneous Dirichlet data on the left, upper and lower boundaries and a volume force
$f$ in such a way that a manufactured solution solves the system.

To construct an analytical solution, 
we define the velocity field $v$ as curl of the scalar function $\psi=k(x,y)^2 (x-1)^3$, where $k(x,y) = (x-x_0)^2 + (y-y_0)^2 -r^2$, 
and choose the pressure in such a way that the \textit{do-nothing condition} holds on the right boundary:
\begin{eqnarray}\label{manuSol}
\begin{aligned}
 v_1 = \partial_y \psi &=4 k(x,y) (x-1)^3 (y-y_0)  \\
 v_2= -\partial_x \psi &= -4 k(x,y) (x-1)^3 (x-x_0) - 3 k(x,y)^2 (x-1)^2 \\
 p= \partial_x v_1 =\partial_{xy} \psi &= 8(x-x_0) (x-1)^3 (y-y_0) + 12  k(x,y) (x-1)^2 (y-y_0).
 \end{aligned}
\end{eqnarray}
In this section, we set $x_0=y_0=0$.

In order to study the effect of the stabilisation term in the anisotropic cells, we set $\Omega_h^{\text{aniso}}=\Omega_h$, which
means that the stabilisation term proposed in this paper
reduces to
\begin{align*}
 S(p_h, \psi_h) = S_h^{\text{aniso}}(p_h,\psi_h) 
 &:= \gamma H^2 \sum_{e \in {\cal E}_h} 
 \int_e \{h_n \nabla p_h\cdot \nabla \psi_h\}_e \, do.
\end{align*}
We will compare the effect of this stabilisation to the standard CIP stabilisation term consisting of jump terms only
\begin{align*}
 S_{cip}(p_h, \psi_h) = \gamma \sum_{e \in {\cal E}_h} h_{\tau}^3 \int_e [\partial_n p_h] [\partial_n \psi_h]\, do.
\end{align*}
Moreover, we consider different cell weights for the two terms. For the anisotropic stabilisation
we consider a variant using only local cell sizes (see Remark~\ref{rem.stabterm})
\begin{align*}
 S_2(p_h, \psi_h) = 
 &:= \gamma \sum_{e \in {\cal E}_h} 
 \int_e \{h_n^3 \partial_n p_h\cdot \partial_n \psi_h\}_e + \{h_n h_{\tau}^2 \partial_n p_h\cdot \partial_n \psi_h\}_e \, do.
\end{align*}
This is the usual
weighting for stabilisation on anisotropic elements, see e.g. Braack \& Richter~\cite{BraackRichter2005Enumath}. 


\begin{table}
  \centering
  {\sfrei
 \begin{tabular}{r|c|c|ccc|ccc} 
 &\multicolumn{1}{c|}{\;\;$\|\nabla v-v_h\|_{L^2}$\;\;}&\multicolumn{1}{c|}{\;\;$\|v-v_h\|_{L^2}\;\;$}
 &\multicolumn{3}{c|}{$\|p-p_h\|_{L^2}$} &\multicolumn{3}{c}{$\|\nabla(p-p_h)\|_{L^2}$}\\
$H$  &\multicolumn{1}{c|}{$S$}  &\multicolumn{1}{c|}{$S$} &\multicolumn{1}{c}{$S$}
&\multicolumn{1}{c}{$S_2$} &\multicolumn{1}{c|}{$S_{cip}$} &\multicolumn{1}{c}{$S$} &\multicolumn{1}{c}{$S_2$} &\multicolumn{1}{c}{$S_{cip}$}\\
\hline
 $1/4\;$  &$20.55$ &$1.27$  &$14.04$   &$14.04$    &$41.23$ &$123.5$    &$3073.6$ &$\;\;535.9$\\
 $1/8\;$  &$10.14$ &$3.20\cdot 10^{-1}$    &$\;\;4.60$            &$\;\;5.86$ &$22.06$ &$\;\;79.8$ &$3910.3$ &$\;\;595.9$\\
 $1/16$   &$\;\;5.02$ &$8.00\cdot 10^{-2}$ &$\;\;1.46$            &$\;\;2.43$ &$16.56$ &$\;\;52.6$ &$4336.1$ &$\;\;912.3$\\
 $1/32$   &$\;\;2.50$ &$2.00\cdot 10^{-2}$ &$\;\;0.47$            &$\;\;1.07$ &$10.82$ &$\;\;35.7$ &$4534.5$ &$1197.9$\\
\hline
Estim.&\;\;1.02 &2.00 &\;\;1.62 &\;\;1.26 &\;0.69 &\;\;0.61 &-0.17 &-0.43\\
\end{tabular}}

\caption{\sfrei $L^2$- and $H^1$-norm errors of velocity and pressure for the three different 
stabilisation variants on an anisotropic grids with anisotropies alternating from 1 to 1000.
The velocity norm errors do not show significant differences for different stabilisations
and are therefore only shown using the stabilisation $S$.
We estimated the convergence order by a least squares fit 
of the function $e(h) =ch^{\alpha}$.\label{tab.anisoterms}}
\end{table}

\begin{figure}[t]
        \centering
            \includegraphics[width=0.3\textwidth]{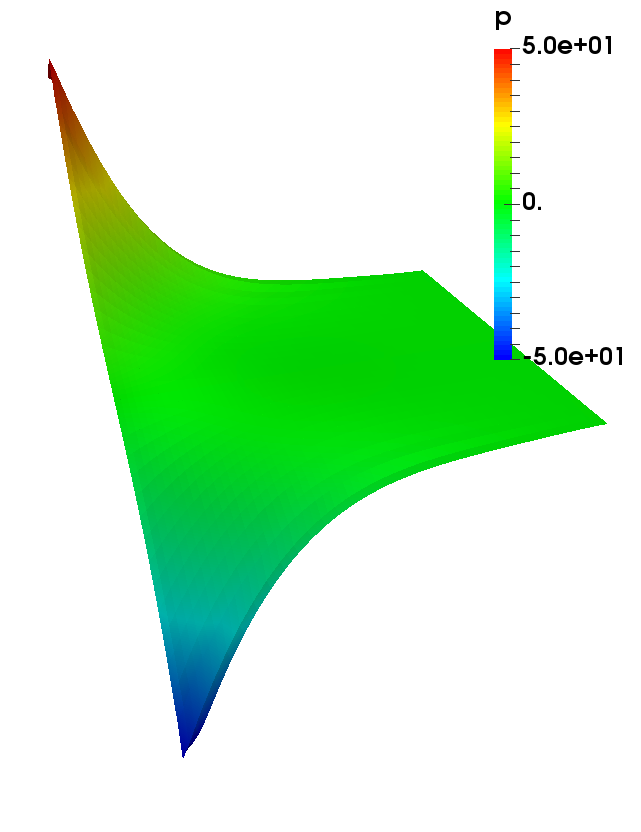}\hfil
            \includegraphics[width=0.3\textwidth]{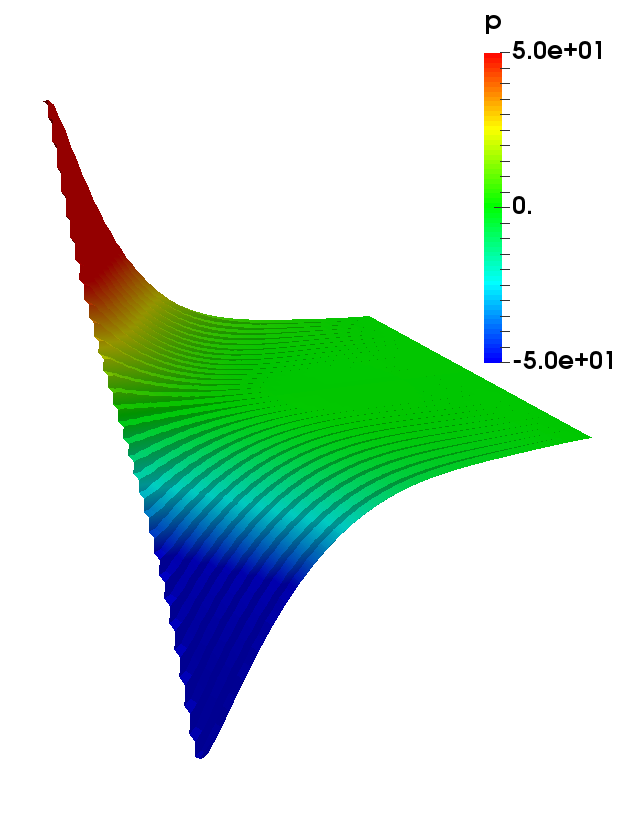}\hfil
            \includegraphics[width=0.3\textwidth]{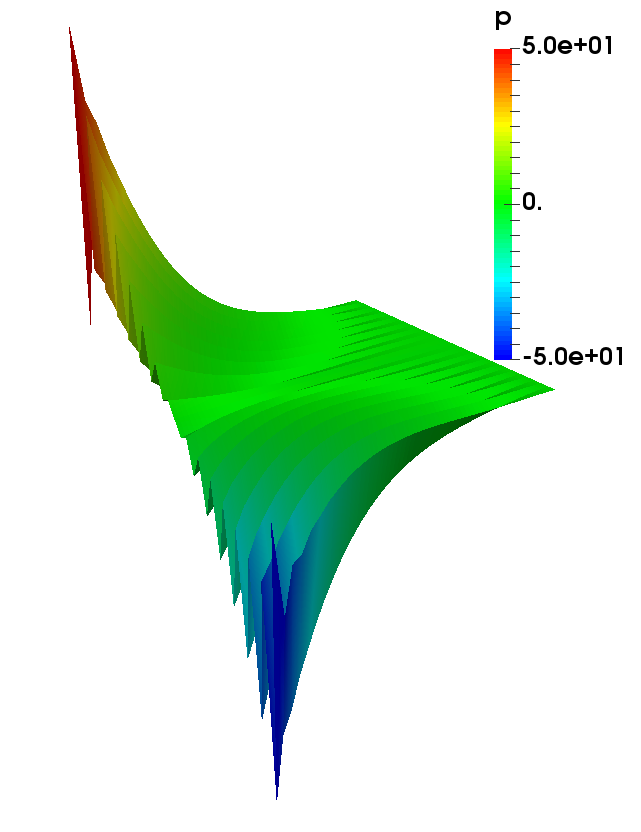}
       \caption{\label{fig.pressplot} Visualisation of the pressure variable over $\Omega$ for the three different stabilisations
       $S, S_2$ and $S_{cip}$ (from left to right) on the mesh with $H=1/16$.}
\end{figure}

In Table~\ref{tab.anisoterms}, we show the $L^2$ and $H^1$-norm errors of the velocities and of pressure on four different meshes. The stabilisation parameter 
has been chosen $\gamma=10^{-2}$ for $S$
and by a factor of $4$ larger for $S_2$ and $S_{cip}$, as on regular cells we have $h_n\approx h_{\tau} \approx H/2$. 
The velocity norm errors do not show significant differences 
for the different stabilisations. Therefore we show only the values 
for the anisotropic stabilisation $S$. The convergence rates for the velocities are as expected. 

Concerning the pressure approximation the situation is different. We observe only slow convergence for the standard CIP stabilisation $S_{cip}$ in the 
$L^2$-norm of pressure, especially on the finer meshes. Changing the weights from $h_{\tau}$ to $h_n$ or $H$ or choosing a larger parameter for 
$\gamma$ did not lead to considerable improvements.

The anisotropic stabilisations, on the other hand, seem to converge even faster than linearly, which would be expected 
from the analysis, as the averages are used everywhere ($\Omega_h^{\text{aniso}}=\Omega$). On the finer meshes, we see a clear advantage 
of the weighting used in the analysis ($S$) compared to using local cell sizes ($S_2$). For this weighting we observe 
even convergence in the $H^1$-seminorm error of the pressure, which increases for $S_2$ and $S_{cip}$.}

{\sfrei
The reason for the different convergence behaviours becomes clear, when we plot the pressure solution over $\Omega$ for the three stabilisations, 
see Figure~\ref{fig.pressplot}
for $H=1/16$. For $S_{cip}$ we observe wild oscillations, which shows that the standard interior penalty stabilisation is not suitable
to control the pressure on this anisotropic mesh.
Smaller oscillations are visible for the term $S_2$, that are due to  to wrong scaling of the derivatives. The stabilisation $S$ leads in contrast
to a smooth behaviour 
of $p$.
}


\begin{figure}
\centering
\begin{minipage}{0.28\textwidth}
\includegraphics[width=\textwidth]{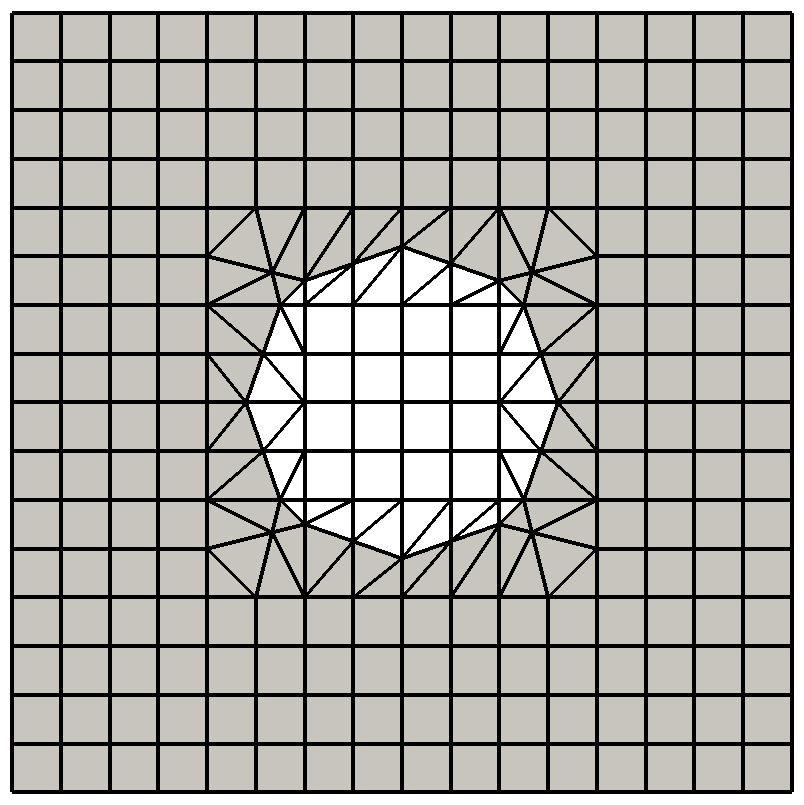}
\end{minipage}
 \caption{\label{fig.press_mesh}Illustration of the coarsest mesh  {\sfrei used for the numerical example of Section 6.2 for $x_0=0$}. 
 The domain $\Omega_{h,f}$ is visualised in grey.}
\end{figure}

\subsection{Example 2: Different kind of anisotropies within the locally modified finite element method}
\label{sec.num_stat}

{\sfrei Next, we show that the proposed pressure stabilisation can be used in combination with the 
\textit{locally modified finite element method} to approximate curved boundaries with all kinds of arising anisotropies.
}
%
%

{\lang To define the geometry we extract an inner circle of radius 
$r=0.4$: $\Omega = (-1,1)^2 \setminus B_{0.4}(x_0,y_0)$ from the unit square}. The boundary $\partial\Omega$ is a mixture
of a polygon and a smooth boundary, such that the theoretical results hold true (cf.$\,$Remark~\ref{rem.smoothDomain}).
We discretise the unit square with a {\sfrei uniform} patch mesh and resolve the {\lang circular boundary by means of} 
the \textit{locally modified finite element} method. We define the set ${\cal T}_h^{\text{aniso}}$ 
as the union of all patches that are cut by the {\lang circle}.
A sketch of a coarse mesh for $(x_0,y_0)=(0,0)$ is given in Figure~\ref{fig.press_mesh}. 
{\sfrei While this mesh is quite isotropic, strong anisotropies will arise when we change the horizontal position $x_0$ 
of the midpoint of the circle.
We use again the manufactured solution (\ref{manuSol}) and the data and boundary conditions specified in the previous example. 
Additionally, we impose homogeneous Dirichlet conditions on the boundary of the circle.}

%

First, we consider the case that the
midpoint of the circle coincides with the origin $x_0=y_0=0$. For ease of implementation, we extend $v$ by zero
in the inner circle and use a harmonic extension of the pressure there.
{\lang We use the stabilisation} $S$ defined and analysed in this work. {\sfrei As in the previous example, we define a
second
stabilisation} $S_2$ that uses the local cell sizes $h_n$ and $h_{\tau}$ as weights instead of $H$.
Precisely, we define
\begin{eqnarray}\label{newpressstab}
\begin{aligned}
 S_2(p_h,\psi_h) := &\gamma_i \sum_{e \in {\cal E}_h^{\text{aniso}}} \int_e 
 \big\{h_n (h_n^2 \partial_n p_h \partial_n \psi_h + h_{\tau}^2\partial_{\tau} p_h \partial_{\tau} \psi_h)\big\}_e \, do 
 \\
 &+\,\gamma_0 \sum_{e \in {\cal E}_h^0} 
 \int_e \left\{h_n\right\}_e^3 [\nabla p_h]_e\cdot [\nabla \psi_h]_e \, do.
\end{aligned}
\end{eqnarray}
Note that for the regular edges in ${\cal E}_h^0$, it holds $h_n \sim H$ and furthermore, 
the jump of the tangential derivatives vanishes. As discussed in Remark~\ref{rem.stabterm}, 
we are not able to show stability for this stabilisation
{\sfrei on the unstructured anisotropic 
mesh arising from the \textit{locally modified finite element method}}.

Finally, we remark that in our implementation, we neglect the jump terms over 
outer patch edges, as they would introduce additional 
couplings in the system matrix. This is not the case for the mean value terms, which have 
to be considered on all edges $e\in {\cal E}_h^{\text{aniso}}$.

\begin{table}
  \centering
 \begin{tabular}{r|rr|cc|rr} 
 &\multicolumn{2}{c|}{$\|\nabla v-v_h\|_{L^2}$}&\multicolumn{2}{c|}{$10^2 \cdot\|v-v_h\|_{L^2}$}
 &\multicolumn{2}{c}{$\|p-p_h\|_{L^2}$}\\
$H$  &\multicolumn{1}{c}{$S$}  &\multicolumn{1}{c|}{$S_2$} &\multicolumn{1}{c}{$S$}
&\multicolumn{1}{c|}{$S_2$} &\multicolumn{1}{c}{$S$} &\multicolumn{1}{c}{$S_2$}\\
\hline
 $1/4\;$  &$18.05$ &$18.04$ &$74.7$&$74.5$ &$3.33$ &$2.82$ \\
 $1/8\;$  &$9.06$ &$9.06$  &$18.7$ &$18.6$ &$1.09$ &$0.91$ \\
 $1/16$   &$4.52$ &$4.52$  &$4.67$ &$4.67$ &$0.36$ &$0.29$ \\
 $1/32$   &$2.26$ &$2.26$  &$1.17$ &$1.17$ &$0.12$ &$0.10$ \\
\hline
Estim.&0.99&0.99 &1.99 &1.99 &1.60 &1.63\\
Expect. &\multicolumn{2}{c|}{1.00}&\multicolumn{2}{c|}{2.00}  &\multicolumn{2}{c}{1.00}\\
\end{tabular}

\caption{$L^2$- and $H^1$-norm of the velocity and $L^2$-norm error of the pressure for {\lang two} different 
stabilisation variants. We estimated the convergence order by a least squares fit 
of the function $e(h) =ch^{\alpha}$
and show the expected convergence rates from Theorem~\ref{theo.conv}.\label{tab:conv_press}}
\end{table}

In Table~\ref{tab:conv_press}, we show the $L^2$- and the $H^1$-norm error of the velocity
as well as the $L^2$-norm error of the pressure for the {\lang two} stabilisations on four different meshes. 
Furthermore, we show an estimated convergence order
based on the calculations.
The stabilisation parameter is chosen $\gamma_i=\gamma_0=2.5\cdot10^{-3}$ for $S$ and again by a factor of 4 larger for $S_2$.

While the velocity errors are almost identical for both stabilisations, the pressure error
is slightly smaller for $S_2$. The convergence behaviour of the velocity norms coincides almost 
perfectly with the theoretical results for 
$S$ given above. The $L^2-$norm of the pressure 
converges with a higher order $\alpha\geq1.5$ for 
both stabilisations, while we had only shown first order convergence
in Theorem~\ref{theo.conv}. 
This can be explained by means of super-convergence effects due the structured grid 
in the sub-domain $\Omega_h^0$. The errors seen here
are essentially a combination of interpolation errors and the error contribution from 
{\lang the} non-consistency of the stabilisation
term in $\Omega_h^{\text{aniso}}$.
For the $L^2$-norm of the pressure, the latter is dominant and restricts 
the convergence order to ${\cal O}(H^{3/2})$.


In order to study the effect of different anisotropies, we move the midpoint of 
the circle {\lang next} in intervals of 
$10^{-3}$ up to $x_0=0.249$ to the right. This covers all kinds of anisotropies, 
as for $x_0=0.25$ the midpoint moves 
by exactly one patch on the coarsest grid. {\sfrei Exemplarily we show in Figure~\ref{fig.aniso} some of the most 
anisotropic cells that arise, with a maximum aspect ratio of 
$1893.9$. Moreover, we give some details of the maximum anisotropies for four different positions $x_0$
in Table~\ref{tab.aniso}.}

\begin{figure}[t]
{\sfrei
  \centering
            \includegraphics[width=0.27\textwidth]{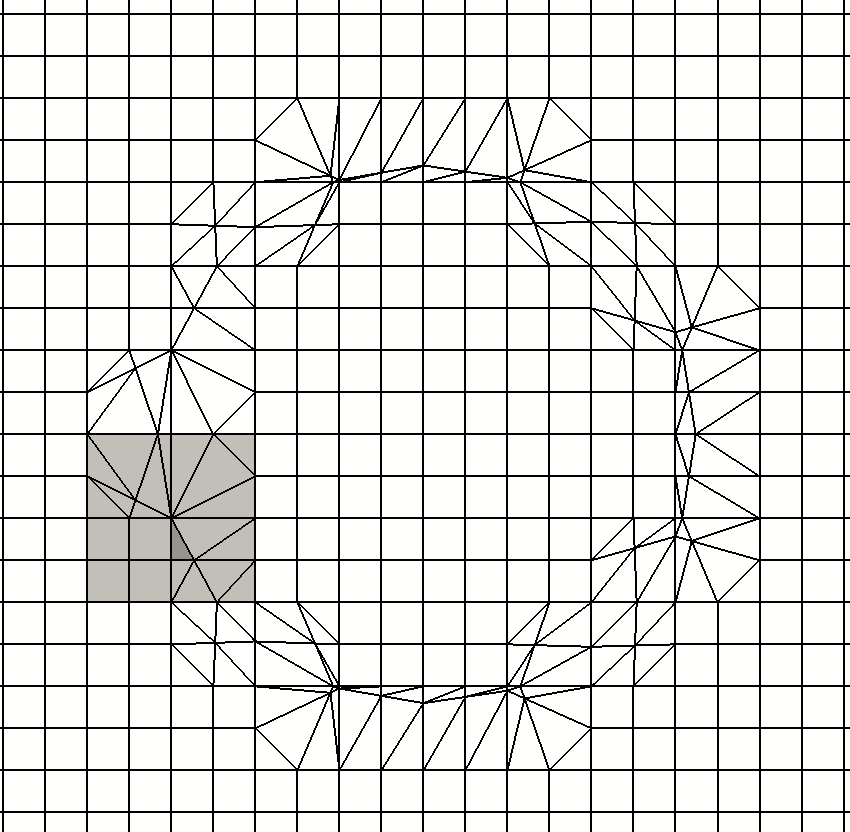}\hfil
            \includegraphics[width=0.27\textwidth]{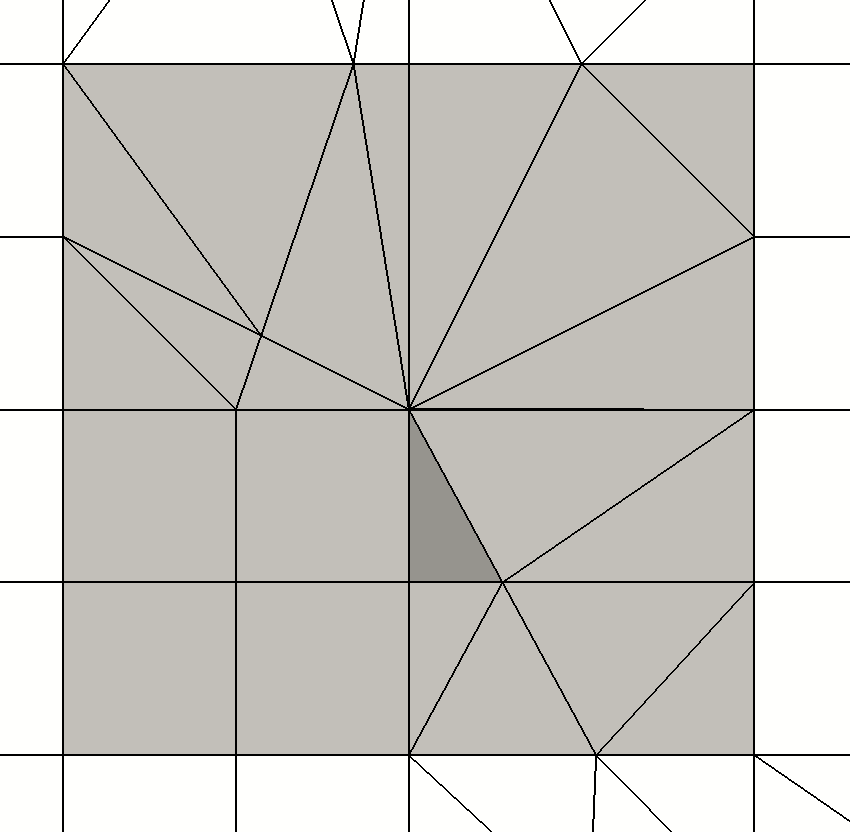}\hfil
            \includegraphics[width=0.27\textwidth]{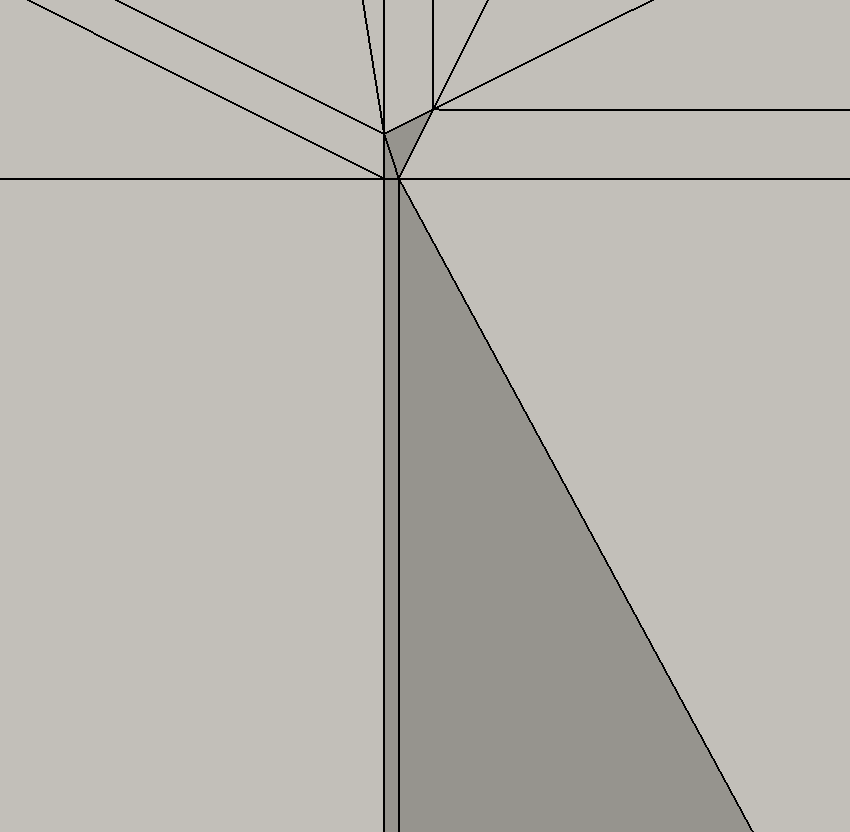}
       \caption{\label{fig.aniso} {\sfrei Visualisation of the mesh anisotropy for $H=1/8$ and $x_0=0.006$. From left to right, we
       zoom twice around the light gray and the dark gray area, respectively. The aspect ratio of the vertically stretched gray
       triangle that is only visible in the second zoom on the right is around $1893.9$.}} 
       }
\end{figure}

\begin{table}
{\sfrei
\begin{center}
     \begin{tabular}{c|ccc|cc|c}
       \toprule
        $x_0$ &$|K_{\max}|$ &$|K_{\min}|$ & $\frac{|K_{\max}|}{|K_{\min}|}$ &  $|e_{\text{max}}|$ &  $|e_{\text{min}}|$ 
        & $\max\limits_{K\in\mathcal{T}_h} \kappa_K$ \\ \hline
        0 &$4.99\cdot 10^{-3}$ &$2.28\cdot 10^{-5}$ &$2.20\cdot 10^2$ &$1.36\cdot 10^{-1}$ &$3.97\cdot 10^{-3}$ &$1.58\cdot 10^1$ \\
        0.006 &$5.85\cdot 10^{-3}$ &$1.65\cdot 10^{-9}$ &$3.55\cdot10^6$ &$1.40\cdot 10^{-1}$ &$3.30\cdot 10^{-5}$ &$1.89\cdot 10^3$\\
        0.015 &$5.84\cdot 10^{-3}$ &$6.98\cdot 10^{-9}$ &$8.37\cdot10^5$ & $1.45\cdot10^{-1}$ &$7.20\cdot10^{-5}$&$8.85\cdot10^2$ \\
       0.045 &$4.93\cdot 10^{-3}$ &$1.86\cdot 10^{-4}$ &$2.66\cdot10^1$ &$1.41\cdot 10^{-1}$ &$1.25\cdot 10^{-2}$ &$5.12$\\
       \bottomrule
     \end{tabular}
   \end{center}
     \caption{\label{tab.aniso} {\sfrei Properties of the mesh $\mathcal{T}_h$ for $H=1/8$ for four different positions $x_0$. 
     In columns 2 to 4, we show the area of the largest and the
     smallest element $|K_{\max}|$ and $|K_{\min}|$ and their ratio; in columns 5 and 6 the size of the largest and smallest 
     edge $|e_{\text{max}}|$ and $|e_{\text{min}}|$. 
     Finally, in column 7 the biggest aspect ratio $\kappa_K=\frac{|e_{K,\text{max}}|}{|e_{K,\text{min}}|}$ 
     of all elements $K\in \mathcal{T}_h$ is shown. For $x_0=0.006$ and $0.015$, two very anisotropic grids 
     emerge, while for $x_0=0.045$ the grid is almost isotropic.}}}
 
\end{table}

In Figure~\ref{fig.H1p} (left sketch), we plot the $H^1$-norm of the pressure over $x_0$ for 
the stabilisation term $S$ and for the four different meshes. The norm increases uniformly when the circle moves to the right as the
analytical solution $p$ increases. We do not observe any instabilities on any of the four grids. 
{\sfrei This shows in particular that the observed convergence behaviour for $x_0=0$ in Table~\ref{tab:conv_press} is obtained on the 
more anisotropic grids  for $x_0>0$ as well.}

In the right sketch, we compare the two different stabilisations on the second-coarsest mesh with $H=1/8$. Again, we do not
observe any oscillations. 

\begin{figure}[t]
\centering
\begin{minipage}{0.35\textwidth}
 \hspace{-0.8cm}
  \includegraphics[width=1.25\textwidth]{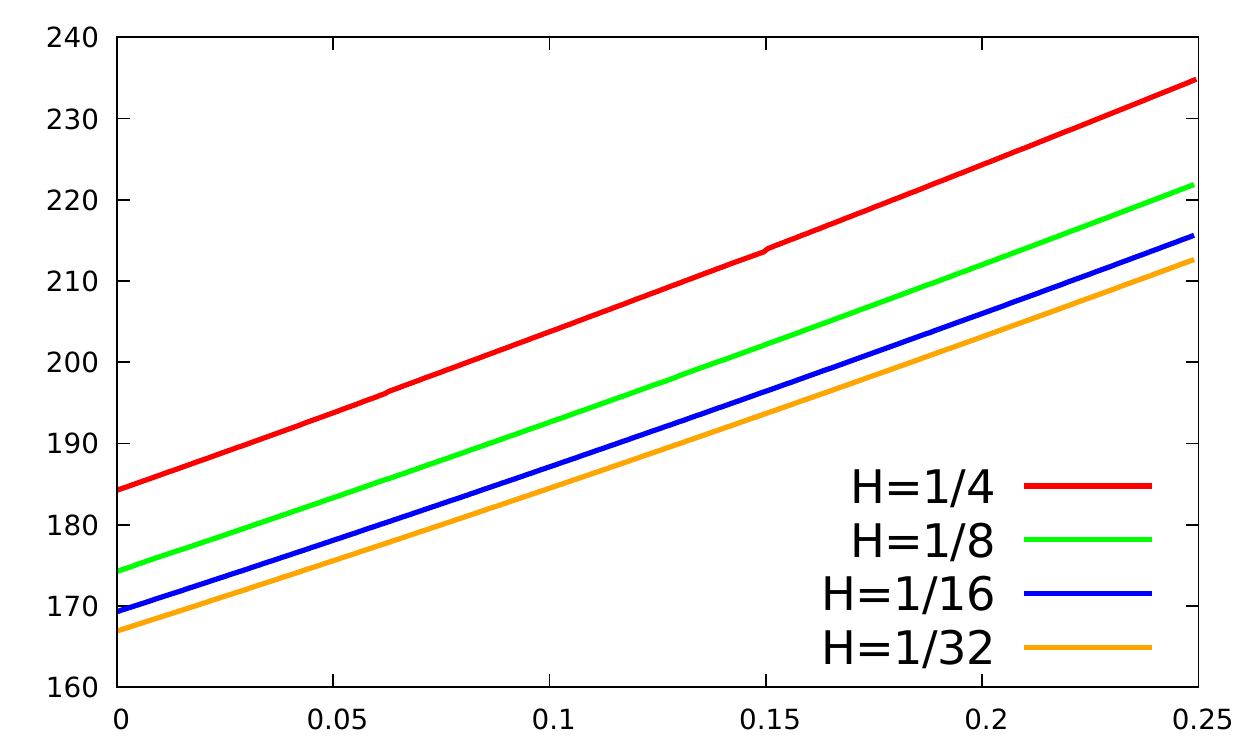}
  \end{minipage}
  \hfil
  \begin{minipage}{0.35\textwidth}
  \hspace{-0.8cm}
  \includegraphics[width=1.25\textwidth]{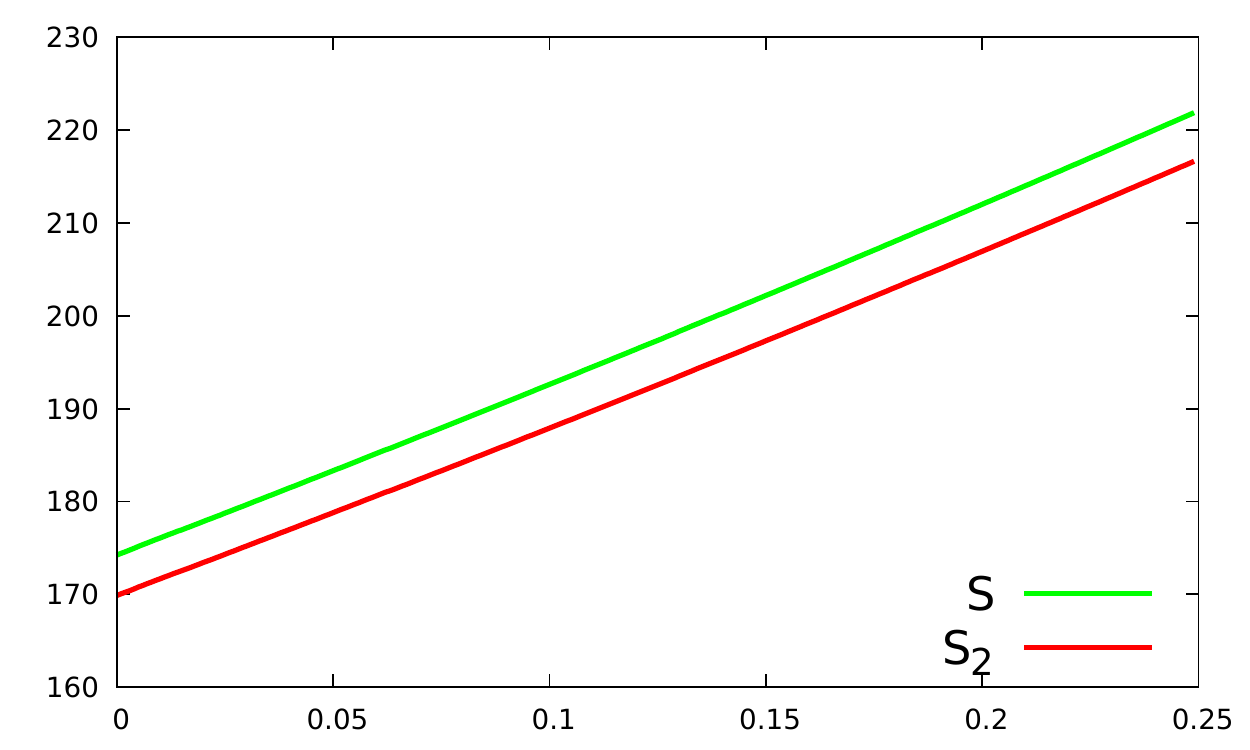}
  \end{minipage}
 \caption{\label{fig.H1p} $H^1$-norm of the discrete pressure for different positions 
 $(x_0,0)$ of the midpoint of the circle plotted over $x_0$.
 \textit{Left:} Stabilisation $S$ on different mesh levels, \textit{right:} Both 
 stabilisations for the mesh with patch size $H=0.125$.}
\end{figure}

\subsection{Example 3: A {\sfrei non-linear} fluid-structure interaction problem}
\label{sec.num_fsi}

To show the capabilities of the approach, we consider a {\sfrei non-stationary and non-linear} fluid-structure interaction problem
{\sfrei with moving interface}. The overall geometry $\Omega$ is the same as in the previous example,
with the difference that the inner ball is now elastic and will be both deformed and moved by the fluid {\lang forces}. 
The sub-domains are denoted by $\Omega_f(t)$ and $\Omega_s(t)$, separated by an 
interface $\Gamma_i(t)$. We consider the incompressible Navier-Stokes equations
{\lang in the fluid domain $\Omega_f(t)$}.
{\lang In} the solid domain $\Omega_s(t)$, 
we impose a hyper-elastic non-linear St.Venant Kirchhoff
material law. Together with the {\lang standard FSI} coupling conditions, the complete set of equations for 
the fluid velocity $v_f$, the pressure $p_f$, the solid displacement $u_s$ and the solid velocity $v_s$ reads 
in Eulerian coordinates
\begin{align}
      \begin{split}
      \rho_f \partial_t v_f + \rho_f (v_f\cdot\nabla) v_f - \operatorname{div} \sigma_f
       &=\rho_f f \\
       \operatorname{div} v_f &=0 \; \end{split}\quad\;\Bigg\} \quad\text{in } \; \Omega_f(t),\notag\\
       \begin{split}
       J \rho^0_s (\partial_t v_s + v_s \cdot \nabla v_s) - 
      \operatorname{div} \sigma_s &= J \rho_s^0 f \\
      \partial_t u_s + v_s \cdot \nabla u_s - v_s &= 0 \end{split}\quad\Bigg\}\quad \text{in } \; \Omega_s(t),\label{completeSet}\\ 
      \begin{split}
            v_f &= v_s \\
            \sigma_fn &= \sigma_s n \;\end{split}\quad\;\,\Bigg\}\quad\text{on } \Gamma_i(t).\notag            
\end{align}
   Here, $F=\nabla T = I - \nabla u_s$ denotes the deformation gradient, $J=\text{det } F$ its determinant and the solid and fluid Cauchy stress tensor are given by
      \begin{align*}
      \sigma_s = J F^{-1} \big(2 \mu_s E_s + &\lambda_s \text{tr}(E_s)\big) F^{-T}, 
      \quad E_s = \frac{1}{2} \left(F^{-T} F^{-1} -I\right),\quad
       \sigma_f = \frac{\rho_f \nu_f}{2}\left(\nabla v_f + \nabla v_f^T\right) -p_f I.
      \end{align*}
      The boundary conditions for the fluid are a parabolic inflow profile on the left boundary, the \textit{do-nothing} boundary 
      condition on the right and
      homogeneous Dirichlet conditions on bottom and top. As material parameters, we use the viscosity $\nu_f=1$, the densities $\rho_f=\rho_s=1000$ and the solid Lam\'e parameters 
      $\mu_s=10^4$ and $\lambda_s=4\cdot10^4$. We start with zero initial data and increase the inflow profile gradually until at $t=0.1$ the profile $v^d(y) = 1-y^2$ is reached.
      
      To solve the system of equations, we use the monolithic \textit{Fully Eulerian} approach introduced by 
      $\,$Dunne \& Rannacher~\cite{DunneRannacher}. {\lang For time discretisation we use
      Rothe's method in combination with} a modified dG(0) time-stepping scheme~\cite{FreiRichterTime}. 
      Due to the moving interface the mesh changes from time step to time step.
      To conserve the incompressibility of the discrete solution, the old velocity is projected onto the new mesh by a Stokes projection after each time step, see Besier \& Wollner~\cite{BesierWollner}. 
      For space discretisation, we use the locally modified finite element method for all variables in combination with the analysed pressure stabilisation technique. 
      A detailed derivation and analysis of the methods can be found in~\cite{DissFrei}.
     
      \begin{figure}
        \centering
            \includegraphics[width=0.27\textwidth]{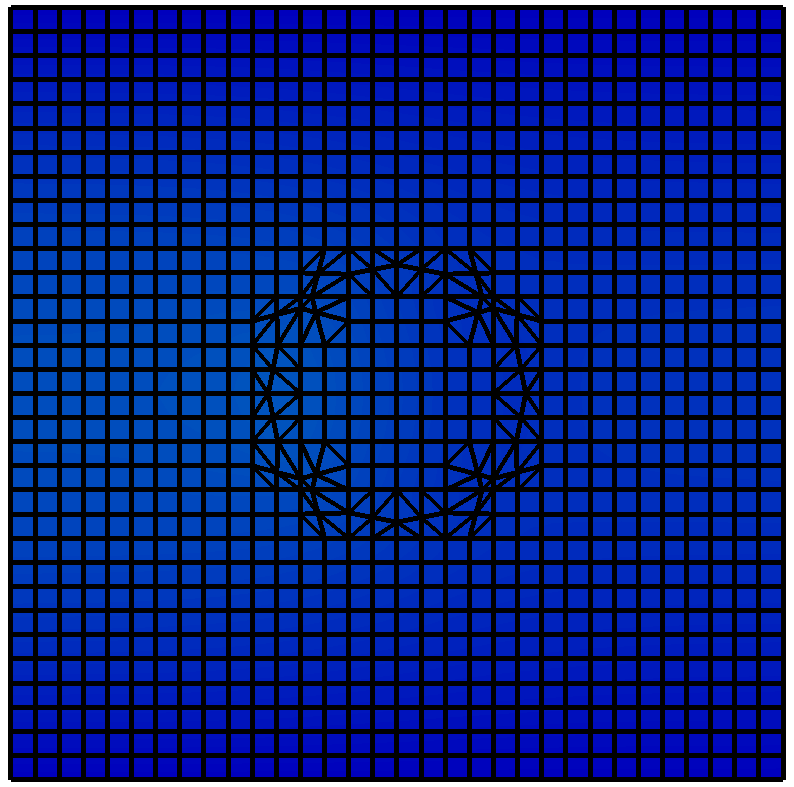}\hfil
            \includegraphics[width=0.27\textwidth]{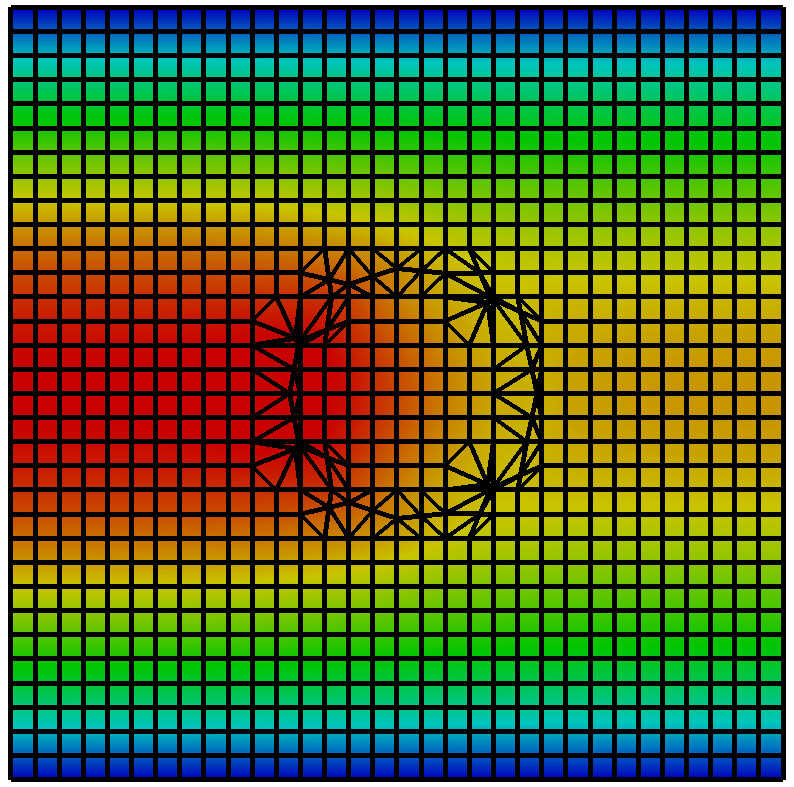}\hfil
            \includegraphics[width=0.27\textwidth]{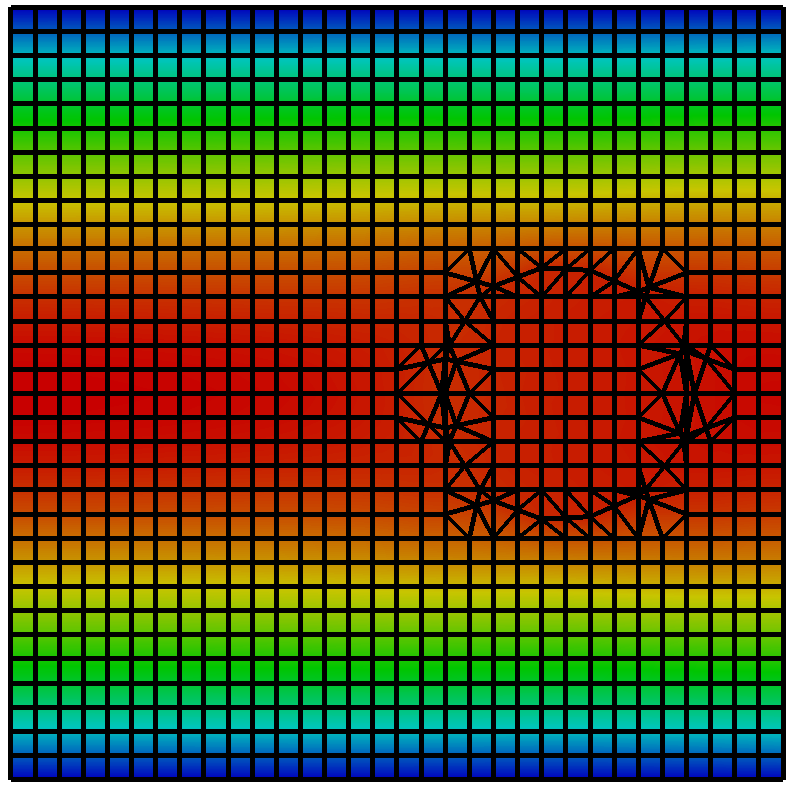}
       \caption{Snapshots of the moving ball on a coarse mesh ($H=1/8$) at times $t=0, 0.1$ and $0.5$. The colouring illustrates the horizontal velocity.\label{fig.movingball}}
      \end{figure}
     
      We study the effect of the stabilisation terms $S$ and $S_2$. 
      The stabilisation parameters are chosen $\gamma_i=10^{-4}$ and $\gamma_0=10^{-2}$
      for $S$ and again by a factor of 4 larger for $S_2$. We use {\lang the
      time step} $k=10^{-2}$ and patch meshes obtained by 4, 5, 6 and 7 global refinements of the
      unit square. 
      For refinement level 5, the resulting mesh including the sub-triangulation that 
      resolves the interface are shown in Figure~\ref{fig.movingball} at three different instances of time. 
      First, the ball is compressed at its left boundary ($t=0.1$, middle), then it starts to move to the right.
      {\sfrei Extremely anisotropic cells occur, as in the previous example.}
      
      In Table~\ref{tab.conv_moving}, we show the $L^2$-norm and the $H^1$-semi-norm of velocity and 
      the $L^2$-norm of the 
      pressure over the fluid domain at time $t=0.5$.
      Again the velocity errors are almost identical for 
      $S$ and $S_2$, while we observe small deviations in the values of the pressure norm.
      All the values converge reasonably well, in most cases even better than predicted. 
      Both pressure stabilisations seem to stabilise similarly well, such that in this example
      no clear advantage for one of the methods can be given.
      
      \begin{table}
  \centering
 \begin{tabular}{r|rr|cc|cc} 
 &\multicolumn{2}{c|}{$\|\nabla v_h\|_{L^2}$}&\multicolumn{2}{c|}{$\|v_h\|_{L^2}$} 
 &\multicolumn{2}{c}{$\|p_h\|_{L^2}$} \\
$H$  &\multicolumn{1}{c}{$S$}  &\multicolumn{1}{c|}{$S_2$} &\multicolumn{1}{c}{$S$} 
&\multicolumn{1}{c|}{$S_2$}&\multicolumn{1}{c}{$S$} &\multicolumn{1}{c}{$S_2$}\\
\hline
 $1/4\;$   &$2.320$ &$2.319$  &$1.354$ &$1.355$ &3813.1 &3707.7\\
 $1/8\;$   &$2.333$ &$2.333$  &$1.356$ &$1.356$ &4608.4 &4597.6\\
 $1/16$    &$2.335$ &$2.335$  &$1.355$ &$1.355$ &4832.8 &4832.6\\
 $1/32$    &$2.335$ &$2.335$  &$1.354$ &$1.354$ &4912.4 &4912.1\\
\end{tabular}

\caption{Behaviour of velocity and pressure norms under mesh refinement for the two different 
stabilisation terms.\label{tab.conv_moving}}
\end{table}
      
Finally, we show a plot of the pressure at time $t=0.5$ on a coarse and a fine mesh in Figure~\ref{fig.press} {\lang for the stabilisation $S$}. 
We see that
on both meshes the pressure is nicely controlled by the stabilisation. On the coarser mesh, however, 
the fine scale behaviour of the 
pressure near the interface $\Gamma_i$ is significantly disturbed.

\begin{figure}
\centering
\includegraphics[width=0.49\textwidth]{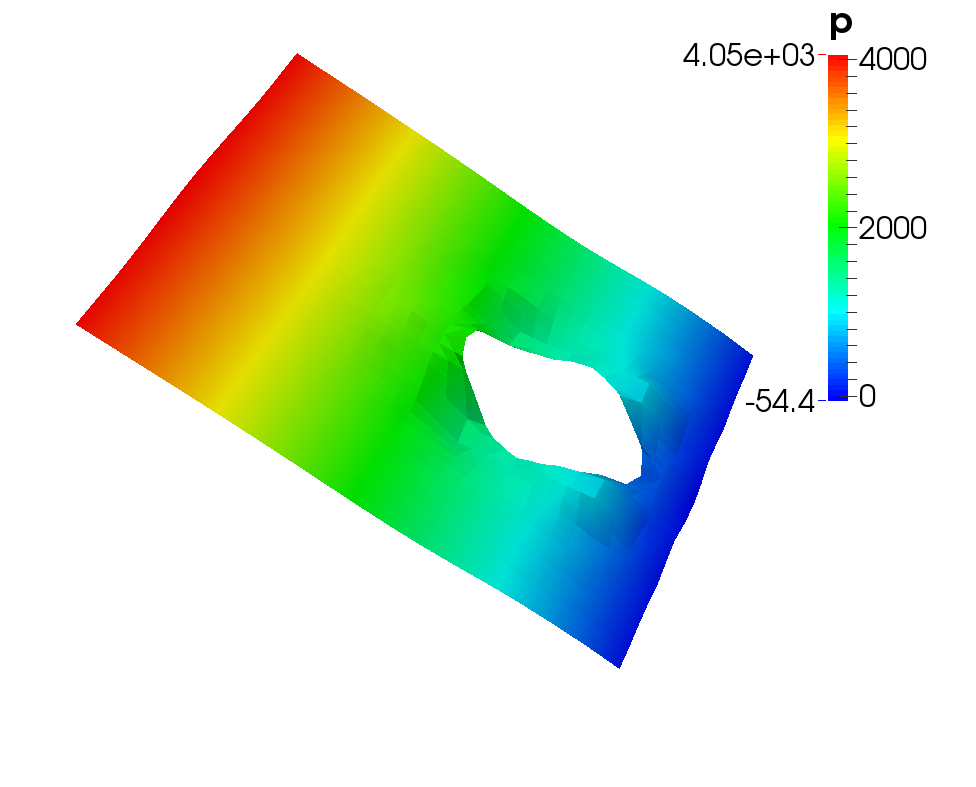}\hfill
\includegraphics[width=0.49\textwidth]{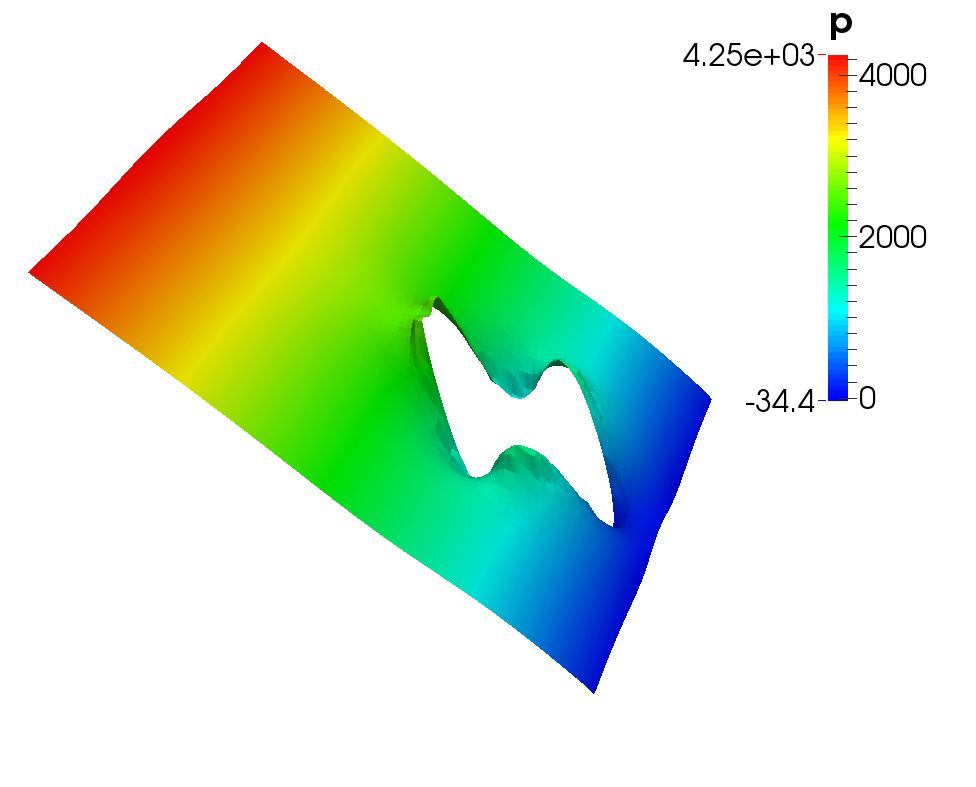}
 \caption{\label{fig.press} Pressure profiles for stabilisation $S$ at time $t=0.5$ on a coarse mesh $H=1/8$ (left) and a fine mesh $H=1/32$ (right).}
\end{figure}





\section{Conclusion}
\label{sec.conclusion}

We have presented a pressure stabilisation scheme that
is able to deal with anisotropic grids without bounded change of anisotropy. 
The approach is especially 
suitable if only a small part of the mesh is anisotropic, which is typical for 
interface problems and problems with complex boundaries. 
{\sfrei Our numerical results show that in contrast to the standard interior penalty pressure stabilisation
the proposed method is able to control the pressure on arbitrarily anisotropic meshes without bounded
changes in anisotropy. A possible extension of this work includes the stabilisation of convection-dominated 
convection-diffusion problems.

Moreover, the pressure stabilisation can be extended to three space dimensions using the corresponding 
stabilisation term on faces instead of edges. The extension of the \textit{locally modified finite element method}
to three space dimensions is 
in principal also possible. The implementation is however subject to future work.
}

\bibliographystyle{plainnat}

\end{document}